\documentclass{amsart}

\usepackage{color}
\newtheorem{thm}{Theorem}[section]

\newtheorem{cor}[thm]{Corollary}

\newtheorem{lem}[thm]{Lemma}

\theoremstyle{definition}

\theoremstyle{remark}
\newtheorem{rem}[thm]{Remark}

\numberwithin{equation}{section}

\newcommand{\abs}[1]{\left\vert#1\right\vert}
\newcommand{\set}[1]{\left\{#1\right\}}

\newcommand{\norm}[1]{\left\Vert#1\right\Vert}

\newcommand{\n}{\nabla}

\newcommand*{\mn}{\mathop{\ooalign{$\|$\cr$\textbf{--}$}}}
\newcommand{\mnorm}[1]{\mn#1\|}

\newcommand{\rbrackets}[1]{\left( #1\right)}

\begin{document}

\title[]
 {On the biharmonic heat equation on complete Riemannian manifolds}

\author{Fei He}

\address{School of Mathematical Science, Xiamen University}

\email{hefei@xmu.edu.cn}

\thanks{}

\thanks{}



\date{}

\dedicatory{}

\commby{}


\begin{abstract}
We study entire solutions of the biharmonic heat equation on complete Riemannian manifolds without boundary. We provide exponential decay estimates for the biharmonic heat kernel under assumptions on the lower bound of Ricci curvature and noncollapsing of unit balls. And we prove a uniqueness criteria for the Cauchy problem. As corollaries we prove the conservation law for the biharmonic heat kernel and a uniform $L^\infty$ estimate for entire solutions starting with bounded initial data.   
\end{abstract}

\maketitle

\begin{center}
 \textit{Dedicated to Professor Peter Li on the occasion of his 70th brithday. }
\end{center}
\section{Introduction}
The heat equation and heat kernel on Riemannian manifolds have been extensively studied and they serve as fundamental tools in geometric analysis, for example, the seminal work on parabolic Harnack estimates \cite{LY1986} by P. Li and S.-T. Yau has been applied and generalized in various literature. On the other hand, the fourth order biharmonic heat equation
\begin{equation}\label{eqn: biharmonic heat equation}
\partial_t u + \Delta^2 u = 0  
\end{equation}
on manifolds is still not well-understood, partly due to the absence of technical tools such as the maximum principle. In this article we consider equation (\ref{eqn: biharmonic heat equation}) on complete noncompact Riemannian manifolds without boundary, and prove some basic properties for entire solutions. 

Recall that (\ref{eqn: biharmonic heat equation}) on the Euclidean space $\mathbb{R}^n$ has a fundamental solution called the biharmonic heat kernel, which (when based at the origin) can be written as:
\begin{equation}\label{eqn: biharmonic heat kernel on Rn}
  b(x, t) = \frac{\alpha_n}{t^{n/4}} F_n(\frac{|x|}{t^{1/4}}),
\end{equation}
where $\alpha_n$ is a normalizing constant and
\[
  F_n(\eta) = \eta^{1-n} \int_0^\infty e^{-s^4} (\eta s)^{n/2} J_{(n-2)/2} (\eta s) ds,   
\]
for $\eta> 0$, $J$ is the Bessel function of the first type. In fact, the behaviour of the biharmonic heat kernel on $\mathbb{R}^n$ is already complicated, and many interesting results depend on its careful analysis, see for example \cite{GG2008} \cite{GG2009} \cite{KL2012} \cite{Wang2012} and the references therein. It is known  that the biharmonic heat kernel on the Euclidean space changes sign infinitely many times \cite{GG2008}, and its absolute value has exponential decay (\cite{KL2012})
\begin{equation}\label{estimate of biharmonic heat kernel on the Euclidean space}
  |b(x, t)| \leq  \frac{C}{t^{n/4}}e^{-\frac{ C |x|^{4/3} }{ t^{1/3}}}.  
\end{equation}

On a complete Riemannian manifold, it follows from standard semi-group theory that there exists a smooth symmetric kernel function $b(x, y, t)$ for the semi-group generated by $-\Delta^2$
densely defined on $L^2(M)$, which is called the biharmonic heat kernel. We prove that $b(x, y,t)$ has 
exponential decay on locally noncollapsed manifolds with Ricci curvature bounded from below. In particular, if the manifold $(M,g)$ is locally noncollapsed and has nonnegative Ricci 
curvature, we have a Li-Yau type estimate. To state it let's first introduce some notations: Denote the volume of a geodesic ball centered at point $p$ with radius $r$ as $V_p(r) = Vol B(p, r)$, let
\begin{equation}\label{definition of the volume ratio}
  \nu_p(r) = \frac{V_p( r)}{\omega_n r^n},
\end{equation}
where $\omega_n$ is the volume of the unit ball in the Euclidean space. On a complete Ricci-nonnegative manifold, $\nu_p(r)\leq 1$ and is nonincreasing by the volme comparison theorem. We say a Ricci-nonnegative manifold has maximal volume growth if $\nu_p(r)$ has a uniform positive lower bound, in this case, it's easy to verify that the lower bound is independent of the base point $p$.  

Choose an increasing function $\psi_p(s)$ such that
\begin{equation}\label{definition of psi}
  \psi_p(s) \leq  \frac{s^4}{ (1- \ln \nu_p(2s))^3} \quad for \quad s > 0,
\end{equation}
note that the RHS is an increasing function when $s$ is large enough, and it is equivalent to $s^4$ when $s \to 0$, so we can take $\psi_p(s)$ such that 
the equation holds in (\ref{definition of psi}) when $s >>1$, and $\psi_p(s) /s^4 \to 1$ as $s \to 0$ . 
\begin{thm}\label{thm: estimate of biharmonic heat kernel nonnegative Ricci curvature}
  Suppose $M$ is a complete noncompact Riemannian manifold with $Ric (x) \geq 0$. Then the biharmonic heat kernel satisfies for any $p,q\in M$and  $t> 0$,
 \[
 | b(p,q,t) | \leq \frac{C(n, \mu) } {\sqrt{V_p(t^{\frac{1}{4}}) V_q(t^{\frac{1}{4}})}} \Psi_p(t) \Psi_q(t) e^{-\frac{c(n) d(p,q)^{4/3}}{t^{1/3}} },
 \] 
 where
 \[
 \Psi_p(t) :=  (1- \ln \nu_p(t^\frac{1}{4}))^{c(\mu)} e^{c(n) ( \frac{\psi_p^{-1}(t)^4}{t} )^\frac{1}{3}},
 \]
$\mu = n$ when the dimension $n \geq 3$ and $\mu > 2$ when $n = 2$.
\end{thm}
\begin{rem}
The quantity $\Psi_p(t)$ can be estimated if we have more information about the volume growth. 

(i)  If $Ric \geq 0$ and $\nu_p(r) \geq v > 0$, $\forall p \in M, r> 0$, we can take $\psi_p(s) = C(v)s^4$, hence $\Psi_p(t) \leq C(n, \mu, v)$ for all $t> 0$. Then we have 
 \[
  | b(p,q,t) | \leq \frac{C(n, \mu,v) } {\sqrt{V_p(t^{\frac{1}{4}}) V_q(t^{\frac{1}{4}})}}e^{-\frac{c(n) d(p,q)^{4/3}}{t^{1/3}} },
\]
which is in the same form as the estimate of the biharmonic heat kernel on the Euclidean space (\ref{estimate of biharmonic heat kernel on the Euclidean space}). 

(ii) If we only have $Ric \geq 0$, by a well-known result of Yau \cite{Yau1976}, we know that $V_p(r) \geq c r$ for some constant $c>0$, hence $1 - \ln \nu_p(r)$ has at most logarithmic growth in $r$. We have taken $\psi_p(s)$ to be equal to the RHS in (\ref{definition of psi}) when $s$ is large enough, thus $\Psi_p(t)$ has at most polynomial growth in $t$ as $t \to \infty$. 
\end{rem}

Our estimates of the biharmonic heat kernel extend the domain of its corresponding semigroup to functions with certain growth condition, depending on the rate of volume growth. This answers partially the question of existence for the Cauchy problem of (\ref{eqn: biharmonic heat equation}).

Next we consider the uniqueness for solutions of the Cauchy problem of (\ref{eqn: biharmonic heat equation}). It is known to experts \cite{KL2012} that solutions bounded by $Ce^{C|x|^{4/3}}$ on $\mathbb{R}^n$ are unique. Another uniqueness theorem on $\mathbb{R}^n$ under different assumptions was proved in \cite{SW2016}, where the authors also provided an example of nonuniqueness generalizing \cite{T1935}. In this article we prove the following:

\begin{thm}\label{uniqueness theorem}
  Let M be a complete Riemannian manifold with $Ric(x) \geq - K(r(x))$, where $r(x)$ is the distance to a fixed point $p \in M$, and $K(r)$ is a nondecreasing function satisfying
  \[
  \int_1^\infty \frac{1}{K(r)^\frac{3}{2}} dr = \infty. 
  \] 
 Let $u$ be a solution of the biharmonic heat equation on $M \times (0, T]$. Suppose $u(x,0) = 0$ in $L^2_{loc}$ sense.  If there exist a constant $a>0$, such that 
 \begin{equation}\label{uniqueness class}
 \int_0^T t^a \int_{B(p,R)} u^2 \leq e^{L(R)},\quad R >>1,
 \end{equation}
 where $L(r)$ is a nondecreasing function satisfying
 \[\int_1^\infty \frac{r^3}{L(r)^3} = \infty.\]
  Then $u(x, t) \equiv 0$ for $(x, t) \in M \times (0,T]$. 
\end{thm}

For example, the non-integrablility assumptions of Theorem \ref{uniqueness theorem} are satisfied for $K(r) = r^\frac{2}{3}$ and $L(r) = r^\frac{4}{3}$. As a direct corollary of the uniqueness and the exponential decay estimate Theorem \ref{exponential decay estimate with Ricci lower bound}, we have the conservation law for the biharmonic heat kernel under some geometric assumptions:
\begin{cor}\label{conservation law}
  Suppose a complete Riemannian manifold $M$ satisfies $Ric \geq -K$ and $Vol (B(x, 1)) \geq v > 0$, $\forall x \in M$, for some nonnegative constants $K$ and $v$. Then the biharmonic heat kernel satisfies
  \[
  \int_M b(x, y, t) dy = 1  
  \] 
  for any $x \in M$ and $t > 0$.
\end{cor}

Another corollary is the following uniform $L^\infty$ estimate for entire solutions of (\ref{eqn: biharmonic heat equation}) starting with bounded initial data. 
\begin{cor}\label{cor: uniform l-infinity estimate}
  Let $M$ be a complete Riemannian manifold with $Ric \geq 0$ and maximal volume growth $\nu_p(r) \geq v > 0$,  $\forall p \in M$, $r > 0$. Then there exists a constant $C(n, \mu, v)$, where $\mu$ is the same as in Theorem \ref{thm: estimate of biharmonic heat kernel nonnegative Ricci curvature}, such that for any solution $u(x, t)$ of the biharmonic heat equation (\ref{eqn: biharmonic heat equation}) satisfying (\ref{uniqueness class}), and with bounded initial data $ |u(x, 0)|_{L^\infty (M)} < \infty $, we have
  \[
  |u(x, t)|_{L^\infty(M)} \leq C |u(x, 0)|_{L^\infty(M)}   
  \] 
  for any $t > 0$.
\end{cor}

This result can be viewed as a ``maximum principle"
for the biharmonic heat equation. It holds true on any closed manifold, see Corollary \ref{cor: l-infinite estimate on closed manifolds}, where the constant depends also on the diameter of the manifold. However the conclusion does not always hold for arbitrary complete manifolds. In the appendix we provide an example of a complete Riemannian manifold where this property fails, namely, there exists a solution of (\ref{eqn: biharmonic heat equation}) with bounded initial data, satisfying \ref{uniqueness class}, but its $L^\infty$ norm goes to infinity as $t \to \infty$. It might be an interesting problem to look for other geometric conditions on Riemannian manifolds that ensure this property. 

We use integration by parts extensively in the proofs, and since equation (\ref{eqn: biharmonic heat equation}) is fourth order, we need cut-off functions with controlled Laplacian.  In Section 2, based on a construction of  Schoen and Yau \cite{SY1994}, we build distance-like functions and cut-off functions with controlled Laplacian depending on the Ricci lower bound, which will be used throughout the rest of the article. In Section 3, we derive $L^2$ estimates and a mean-value inequality for solutions of the biharmonic heat equation. In Section 4 we study the growth of a weighted $L^2$ integral of the solution, we show that it can be made monotone when the weight functions are properly chosen.  In section 5 we use the ingredients developed in the previous sections to generalize a method in \cite{G2009} to equation (\ref{eqn: biharmonic heat equation}), and prove exponential decay estimates for the biharmonic heat kernel on noncollapsed manifolds with Ricci curvature bounded from below, Theorem \ref{thm: estimate of biharmonic heat kernel nonnegative Ricci curvature} is proved as a special case. Theorem \ref{uniqueness theorem} and Corollary \ref{conservation law} are proved in Section 6. The $L^\infty$ estimates will be discussed in Section 7. 

\textbf{Acknowledgement:} This article is dedicated to my former advisor Professor Peter Li for his 70th birthday, I would like to thank Professor Li for his teaching and constant encouragement. 

\section{Construction of distance-like functions and cut-off functions}
In this article, we need distance-like functions with bounded gradient and Laplacian, which will be used in place of the Riemannian distance function. The construction follows from Schoen and Yau \cite{SY1994}, which is sufficient for application in the case $Ric \geq -1 $ uniformly, outside of a ball of certain radius. However, we will consider manifolds where the Ricci curvature is bounded by a function from below, which may go to $- \infty$. We will use a gluing argument to construct a distance-like function with gradient bounded uniformly, and with Laplacian bounded locally depending on the Ricci lower bound. Such a function exists on any complete Riemannian manifold. On a manifold with nonnegative Ricci curvature, we will use the same gluing argument to construct a distance-like function whose Laplacian has almost linear decay, depending on the volume growth rate.  
Here we include the proof of \cite{SY1994} with some slight modifications, and we make clear the dependence of the constants on geometric information which is important for our purpose. 


\subsection{Schoen-Yau distance-like function}

We use $d(p,x)$ to denote the distance between two points $p$ and $x$ on a Riemannian manifold $M$, and use $B(p,R)$ to denote the geodesic ball with radius $R$ centered at $p$. The volume of a set $\Omega$ is denoted by $|\Omega|$. $C(n)$ denotes a constant that depends only on the dimension $n$, which may change from line to line.   

\begin{lem}[Schoen-Yau, Thm 4.2 \cite{SY1994}] \label{distance-like function lemma 1}
Suppose $Ric \geq -K$ on a geodesic ball $B(p,R+1) \subset M $ as a proper subset, where $R >3$,  then there exists a smooth function $f$ defined on $B(p,R)$ such that 
\[
d(p,x) \leq f(x) \leq C(n) d(p,x) + \frac{C(n) + v}{\sqrt{K+1}}, \]
\[ |\n f(x)| \leq C(n), \quad |\Delta f (x)| \leq C(n)\sqrt{1+K},
\]
for $x \in B(p,R)\backslash B(p,2)$, the constant $v = \abs{\log \left( (\sqrt{1+K})^n|  B(p, 1/\sqrt{1+K})| \right)}$. And we can take $R= \infty$, in which case the function $f$ is defined on the entire manifold. 
\end{lem}
\begin{proof}
Let\rq{}s first consider the case $R< \infty$.

Let $\lambda>0$ be a constant to be determined later. For each finite $R > 1$, the following Dirichlet problem has a unique solution denoted as $h_R$.  
\[
\begin{cases}
\Delta h_R = \lambda h_R, & on \quad B(p,R+1) \backslash B(p, \frac{1}{2}); \\
h_R(x) = 1, & x \in \partial B(p,\frac{1}{2}); \\
h_R(x) = 0, & x \in \partial B(p,R+1).
\end{cases}
\]
Let $r(x) = d(p,x)$ and denote $B_R = B(p,R)$. By the maximum principle we know that $0< h_R < 1$ in the interior of $B_{R+1} \backslash B_{\frac{1}{2}}$. By the Cheng-Yau gradient estimate we have 
\[
|\n h_R(x) | \leq C(n)\sqrt{1 + K + \lambda}, \quad x \in \partial B_1.
\]
Let $A$ be a constant to be determined later. Using integration by parts and the Cauchy-Schwarz inequality, we can calculate 
\[
\begin{split}
\lambda \int_{B_{R+1}\backslash B_1} e^{Ar}h_R^2 = & \int_{B_{R+1}\backslash B_1}  e^{Ar} h_R \Delta h_R \\
= & \int_{B_{R+1}\backslash B_1}  - e^{Ar} |\n h_R|^2 - A e^{Ar}h_R \langle \n r, \n h_R\rangle - e^A \int_{\partial B_1}  h_R\frac{\partial h_R}{\partial \nu}\\
\leq &  \frac{A^2}{4} \int_{B_{R+1}\backslash B_1}  e^{Ar} h_R^2 + e^A \int_{\partial B_1}  |\n h_R|.\\
\end{split}
\]
Let $\lambda = \frac{A^2}{4} + 1$, then we have
\[
\int_{B_{R+1} \backslash B_1} e^{Ar} h_R^2 \leq  e^A \int_{\partial B_1} |\n h_R| \leq  e^A |\partial B_1| C(n)\sqrt{1 + K + \lambda}: = \tilde{C}.
\]
For any $x \in B_{R} \backslash B_2$, $y \in B(x,1)$, we have $r(y) \geq r(x) -1$, and
\[
\int_{B(x,1)} h_R^2 \leq \tilde{C}e^{_A} e^{-Ar(x)}.
\]
By the Harnack inequality we have
\[
h_R(x) \leq C(n)e^{C(n)\sqrt{K+\lambda}} h_R(y),  \quad y \in B(x, 1).
\]
By the volume comparison theorem we can show that 
\[
Vol B(x,1) \geq c(n)|B_1|e^{-C(n)\sqrt{K +1} r(x) }.
\]
Hence
\[
h_R^2(x) \leq e^A |\partial B_1| C(n)\sqrt{1 + K + \lambda} |B_1|^{-1} e^{(C(n)\sqrt{K +\lambda} +C(n)\sqrt{K +1} - A) r(x)}, 
\]
where $ x \in B_{R} \backslash B_2 $.

Recall that $\lambda = \frac{A^2}{4} +1$, by elementary inequality,
\[C(n)\sqrt{K + \lambda} \leq C(n)\sqrt{K} + C(n)\sqrt{\lambda} \leq C(n)\sqrt{K} + C(n)^2 + A/2,\]
we can choose $A = \frac{4}{3} ( 3 C(n)\sqrt{K +1} + C(n)^2 +1)$, then
\[
C(n)\sqrt{K +\lambda} +C(n)\sqrt{K +1} - A \leq - \sqrt{1+K}.
\]
Hence by choosing a larger $C(n)$, we have 
 \[
h^2_R(x) \leq C(n)e^{C(n)\sqrt{1+K}} \frac{|\partial B_1|}{ |B_1| }e^{- \sqrt{1+K}r(x)} : =\bar{C} e^{-\sqrt{1+K} r(x)}, \quad x \in B_{R} \backslash B_2. 
\]
Let $\eta$ be a cut-off function on $M$, such that $\eta = 1 $ on $B_{1}$, $0 < \eta < 1$ on $B_{2}\backslash B_1$ and  $\eta = 0$ on $M \backslash B_{2}$.  Define
\[
f_R = \frac{1}{\sqrt{1+K}}\left( -(1-\eta) \log h_R^2 + \log \bar{C}\right). 
\]
By the estimate of $h_R$, when $ r(x) >2$ we have $f(x) \geq r(x)$.

By Cheng-Yau gradient estimate and the choice of $\lambda$, we have
\[
|\n f_R| \leq \frac{2}{\sqrt{1+K}} |\n \log h_R| \leq C(n) \quad on \quad B_R \backslash B_2.
\]
By maximum principle $h_R > h_3>0$ on $B_3 \backslash \bar{B_1}$, hence
\[
f_R|_{\partial B_2} < \frac{1}{\sqrt{1+K}}\left( - 2\log h_3 |_{\partial B_2} + \log \bar{C}\right).
\]
Then the gradient estimate for $f_R$ implies
\[ 
f_R(x) \leq f_R|_{\partial B_2}(x) + C(n) (r(x) -2) \leq \frac{- 2\log \inf_{\partial B_2}h_3+\log \bar{C}}{\sqrt{1+K}}  + C(n ) r(x), \quad x \in B_R \backslash B_2.
\]
To estimate $- \log \inf_{\partial B_2}h_3$, assume $Ric(x) \geq - (n-1)k$ for $k\geq 0$, $x\in B_3\backslash B_{1/2}$,  suppose $\psi(r)$ solves the ODE:
\[
\psi^{\prime \prime} + (n-1)\sqrt{k} \coth \sqrt{k}r \psi^\prime = \lambda \psi, 
\]
\[
\psi(\frac{1}{2}) = 1, \quad \psi(3) = 0.
\]
By the maximum principle $\psi(r) > 0$ for $\frac{1}{2} < r < 3$ and $\psi$ is a nonincreasing function. By the Laplacian comparison theorem the transplantation $\psi(r(x))$ is a subsolution of the equation
\[
\Delta \psi(r(x)) \geq \lambda \psi(r(x)).
\]
Hence by the maximum principle we have 
\[
\inf_{\partial B_2} h_3 \geq \psi(2) >0. 
\]
Clearly the value of $\psi(2)$ depends only on $n$ and $k$ and $\lambda$, hence on $n$ and $K$. So we have
\[
r(x) \leq f_R(x) \leq C(n) r(x) + \hat{C}(n,K) + \frac{\log(|\partial B_1| / |B_1|)}{\sqrt{1+K}}, \quad 2 \leq r(x) \leq R.  
\]

To control the Laplacian, we can calculate when $r(x) > 2$ that
\[
\Delta f_R =\frac{2}{\sqrt{1+K}} \left( -\frac{\Delta h_R}{h_R } + |\n \log h_R|^2 \right)= \frac{2}{\sqrt{1+K}} \left(-\lambda + |\n \log h_R|^2 \right),
\]
which is bounded by a constant $C(n)\sqrt{1+K}$ by the gradient estimate and the choice of $\lambda$. 

It remains to clarify the dependence of the constant $\hat{C}(n,K)$ on $K$. To do this, we can use a scaling argument. Let $\tilde{g} = (1+K)g$, then we have $Ric(\tilde{g}) \geq -1$. Denote the distance function w.r.t. $\tilde{g}$ as $\tilde{r}(x)$, then $\tilde{r}(x) = \sqrt{1+K} r(x)$. By the above argument we can find a function $\tilde{f}_{\sqrt{1+K}R}$ such that
\[
\tilde{r}(x) \leq \tilde{f}_{\sqrt{1+K}R}(x) \leq C(n) \tilde{r}(x)+ \hat{C}(n,1) + \log(|\partial \tilde{B}_1| / |\tilde{B}_1|) ,
\]
\[
|\tilde{\Delta} \tilde{f}_{\sqrt{1+K}R} (x)| \leq C(n).
\]
Note that
\[
\log(|\partial \tilde{B}_1| / |\tilde{B}_1| ) = \log((\sqrt{1+K})^{n-1}|\partial B_{\frac{1}{ \sqrt{1+K}} } | ) - \log(  (\sqrt{1+K})^n|  B_{\frac{1}{ \sqrt{1+K}}}| ),
\]
by the comparison theorem, the above only depend on $n$ and a lower bound of the volume ratio $ (\sqrt{1+K})^n|  B_{\frac{1}{ \sqrt{1+K}}}|$.

Now we define the distance-like function by rescaling,
\[ 
f_R(x) = \frac{1}{\sqrt{1+K}} \tilde{f}_{\sqrt{1+K}R}(x), 
\]
it is then direct calculation to check that $f_R(x)$ satisfies the desired properties.

In the case $R=\infty$, we can take a sequence of finite numbers $R_k \to \infty$ as $k \to \infty$, and the corresponding $h_{R_k}$ is a monotone increasing sequence of uniformly bounded functions which converge to a function $h$ smoothly on any compact set, where the smooth convergence is guaranteed by elliptic regularity. The above estimates for $h_R$ holds for $h$ with the same constants, thus we can define $f$ similarly.
\end{proof}
Alternatively we can restate the above lemma in the following way, which will be convenient to use in the following discussion.
\begin{lem} \label{distance-like function lemma 2}
Suppose $Ric \geq -K$ on a geodesic ball $B(p,R+1) \subset M $ as a proper subset, where $R_0 < R \leq \infty$,  then there exists a smooth function $f$ defined on $B(p,R)$ such that 
\[
d(p,x) \leq f(x) \leq \Lambda(n) d(p,x), \]
\[ |\n f(x)| \leq \Lambda(n), \quad |\Delta f (x)| \leq \Lambda(n)\sqrt{1+K},
\]
for $x \in B(p,R)\backslash B(p, R_0)$, where $R_0 = \frac{C(n) + v}{\sqrt{1+K}}$, $v = \abs{ \ln \left( (\sqrt{1+K})^n|  B(p, 1/\sqrt{1+K})| \right) }$, $\Lambda(n)$ is a constant depending only on $n$. In the case $R= \infty$, the function $f$ is defined on the entire manifold. 
\end{lem}
\begin{rem}
  For any $l > 0$, we can scale the metric by a factor $(\frac{R_0}{l})^2$, apply the above lemma, then scale back. By this scaling argument, we can let $f$ satisfy the above estimates outside $B(p, l)$ for any $l> 0$ instead of $R_0$, however, the upper bound of $\Delta f$ will be dependent also on $v$:
  \[|\Delta f| \leq l^{-1}R_0 \Lambda \sqrt{1+K}.\] 
\end{rem}

\subsection{Construction of cut-off functions}\label{construction of cut-off function}
Suppose $Ric(x) \geq -K(r(x))$, where $K(r)$ is some nondecreasing function of $r$. Then by Lemma \ref{distance-like function lemma 2} there is a constant $\Lambda$ depending on $n$, and a constant $R_0$ depending on the geometry of $B(p,1)$, such that for any $R \geq \Lambda R_0$ and $\rho> 0$, we have a  distance-like function on $B(p,R+\rho)$ satisfying
\begin{equation}\label{equivalence to distance}
r(x) \leq f(x) \leq \Lambda r(x),  \quad when \quad R_0 < r(x) < R+ \rho.
\end{equation}
Note that the construction of $f$ actually depends on $R+\rho$. On $B(p, R_0)$, for the time being $f$ can be defined arbitrarily so that $0 \leq f(x) \leq \Lambda R_0$. 

We\rq{}ll use $f(x)$ instead of the distance function in the following discussion, and we denote
\begin{equation}\label{regularized ball}
D_r = f^{-1}([0, r]).
\end{equation}
By (\ref{equivalence to distance}) we have the relation
\begin{equation}\label{equivalence of balls}
B(p,\Lambda^{-1}r) \subset D_r \subset B(p,r), \quad for \quad \Lambda R_0 < r < R + \rho. 
\end{equation}
Let $\eta(r)$ be a smooth function such that $\eta(r) = 1$ when $ r \leq 1$, $0 < \eta(r) < 1$ when $1< r < 2$ and $\eta(r) = 0$ when $r\geq 2$, and $- 2 < \eta^\prime \leq 0$, $|\eta ^{\prime \prime }| \leq 10$. 

For any $k > 0$ and $\rho > 0$, we define a cut-off function
\begin{equation}\label{cut-off function}
\phi(x) = \eta^k ( 1+\rho^{-1}(f(x)- R)),
\end{equation}
which satisfies
\[
\phi(x) = \begin{cases}
1 , & x \in D_R;\\
\in [0,1], & x \in D_{R+\rho}\backslash  D_R;\\
0,  & x \in M\backslash D_{R+ \rho}.
\end{cases}
\]
By (\ref{equivalence of balls}) we have $\phi = 1$ on $B_{\Lambda^{-1}R}$ and $\phi = 0$ on $M \backslash B_{R+\rho}$. From direct calculation and Lemma \ref{distance-like function lemma 2}, we can verify that $\phi$ has well-controlled gradient and laplacian: 
\begin{equation}\label{gradient of the cut-off function}
|\n \phi| \leq  \rho^{-1}k \eta^{k-1} |\eta^{\prime}| |\n f| \leq \frac{C(n)k}{\rho} \phi^{1-1/k} .
\end{equation}

\begin{equation}\label{Laplacian of the cut-off function}
\begin{split}
|\Delta \phi| =& | \rho^{-2}k(k-1)\eta^{k-2} (\eta^\prime)^2 |\n f|^2 + \rho^{-2}k\eta^{k-1}\eta^{\prime\prime} |\n f|^2 + \rho^{-1}k \eta^{k-1} \eta^{\prime} \Delta f| \\
\leq & \frac{C(n)k^2\sqrt{1+K(R+\rho)}}{\rho} \phi^{1-2/k},
\end{split}
\end{equation}
when $\rho \geq 1$, and 
\[
  |\Delta \phi| \leq \max \set{\frac{C(n)k^2\sqrt{1+K(R+\rho)}}{\rho} \phi^{1-2/k}, \quad \frac{C(n)k^2}{\rho^2} \phi^{1-2/k} }
\]
when $\rho < 1$.

\begin{rem}
Note that we assumed $R> R_0$ in the above construction, this assumption can be dropped using the distance-like function constructed in the next subsection, where the bound for the Laplacian will be in a slightly different form.
\end{rem}
\begin{rem}
 There is a different construction of cut-off functions with controlled Laplacian by Cheeger and Colding \cite{CC1996}, assuming a Ricci lower bound, which does not require largeness of the radius.
\end{rem}

\subsection{Globally defined distance-like functions.}
On every complete Riemannian manifold, there exists some nonnegative nondecreasing function $K(r)$, such that $Ric(x) \geq - K(r)$ for $x \in B(p,r)$, where $p$ is a fixed point. In this section, we first construct a globally defined distance-like function whose Laplacian is controlled by the function $K(r)$, then we will consider the special case that $K(r) \equiv 0$. 

\begin{lem} \label{global distance-like function}
Let $M$ be a complete noncompact Riemannian manifold with $Ric(x) \geq - K(r(x))$, where $K(r)\geq 0$ is a nondecreasing function, $r(x)$ is the distance from $x$ to a fixed point $p$. There exists a function $f$ satisfying
\[
r(x) \leq f(x) \leq \Lambda r(x), \quad |\n f(x)| \leq \Lambda, \]
for all $x \in M \backslash \set{p}$, and 
\[
\quad |\Delta f (x)| \leq \frac{\Lambda R_0 \sqrt{1+K(C(n) R_0)} }{r( x)} \quad for \quad x \in B(p, R_0),
\]
\[  
\quad |\Delta f (x)| \leq \Lambda\sqrt{1+K(4 r(x))} \quad for \quad x \in M\backslash B(p, R_0),
\]
where the constant $\Lambda$ depends only on $n$, the constant $R_0$ depends on $n$, $K(1)$ and $Vol (B(p,1))$ in the form (\ref{formula for R_0}).
\end{lem}

\begin{proof}
By the volume comparison theorem, there are constants depending on $n$ and the Ricci lower bound $Ric \geq - K(1)$ on $B(p,1)$, such that
\[
e^{C(n)K(1)} \geq (1+K(r))^\frac{n}{2} Vol(B(p, (1+K(r))^{-\frac{1}{2}})) \geq e^{- c(n) K(1)} Vol(B(p,1)),
\]
for any $r\geq 0$. So by the explicit fomula for $R_0$ in Lemma \ref{distance-like function lemma 2}, we can choose the constant $R_0$ depending on $n, K(1)$ and $Vol(B(p,1))$ when we apply Lemma \ref{distance-like function lemma 2}, in particular, we can take
\begin{equation}\label{formula for R_0}
R_0 = \frac{C(n) (1+ K(1) + |\ln Vol(B(p,1))| ) }{\sqrt{1+ {K(1)}}}. 
\end{equation}

\textbf{Step 1:}
Let $R_i = (2\Lambda )^i R_0$, $i = 0, 1, 2, ...$, where $R_0$ is from Lemma \ref{distance-like function lemma 2}, Let $B_i = B(p, R_i)$. For each $i$ we can construct a cut-off function $\phi_i$ as in section \ref{construction of cut-off function} with $k = 2$ and $\rho = \Lambda R_i$, such that $\phi_i =1$ on $B_i$, $\phi_i \geq 0$, $\phi_i = 0$ on $M\backslash B_{i+1}$, and 
\[
|\n \phi _i | \leq \frac{C(n)}{R_i}, \quad |\Delta \phi_i| \leq \frac{C(n) \sqrt{1+ K(R_{i+1})}}{R_i}.
\]

Let $f_i$ be the distance function on $B_{i+1}$ from Lemma \ref{distance-like function lemma 2}, so 
\[
d(p,x) \leq f_i(x) \leq \Lambda d(p,x), \]
\[ |\n f_i(x)| \leq \Lambda, \quad |\Delta f_i (x)| \leq \Lambda\sqrt{1+K(R_{i+1})}.
\]
Let $\psi_1 = \phi_1$, $\psi_i = \phi_{i} - \phi_{i-1}$, $i = 2,3...$, note that $\psi_i$ is supported on $B_{i+1} \backslash B_{i-1}$. Define
\[
f = \sum_{i =1}^\infty \psi_i f_i.
\]
For each $x \in M \backslash B(p, R_0)$, there is an $i > 0$ such that $R_i \leq d(p,x) \leq R_{i+1}$. Hence $\psi_j(x) = 0$ unless $j = i$ or $j=i+1$. So the RHS above is a finite sum for each $x$, hence $f$ is well defined. Moreover, for this $x$ we have
\[
f(x) = \psi_i(x) f_i(x) + \psi_{i+1}(x) f_{i+1}(x),  
\]
and $\psi_{i}(x)+\psi_{i+1}(x) = 1$, hence 
\[d(p,x) \leq f(x) \leq \Lambda d(p,x).\]
In particular, $f(x) \leq \Lambda R_{i+1}$.

For the gradient and Laplacian of $f$, we compute
\[
|\n f(x)| = |\n \sum_{j=i, i+1} \psi_j(x) f_j(x)|\leq \sum_{j=i, i+1} \left( |\n \psi_j| f_j + \psi_j |\n f_j|\right)(x) \leq C(n),
\]  
\[
  \begin{split}
|\Delta f|(x) = & | \Delta (\psi_i f_i + \psi_{i+1} f_{i+1})(x)| \\
& \leq\sum_{j = i, i+1} |\Delta \psi_j | f_j + \psi_j |\Delta f_j| + 2 |\psi_j | |\n f_j|\\
& \leq C(n)\sqrt{1+ K(R_{i+2})} \leq C(n)\sqrt{1+K(4 d(p,x))},
  \end{split}
\]
for some constant $C(n)$. Hence, by redefining the constant $\Lambda(n)$, we have shown that
there exists a function $f$ satisfying
\[
d(p,x) \leq f(x) \leq \Lambda d(p,x), \]
\[ |\n f(x)| \leq \Lambda, \quad |\Delta f (x)| \leq \Lambda\sqrt{1+K(4 d(p,x))},
\]
for $x \in M\backslash B(p, R_0)$.

\textbf{Step 2:}
The second step is to extend the definition of $f$ to $B(p, R_0)$.

For any $0 < \rho < 1$, we can scale the metric $g$ and define $g_\rho = \rho^{-2}g$. In the following we will use $\n_{g_\rho}$ to denote the gradient w.r.t. $g_\rho$, and omit the subscript when the metric is $g$. 
  
By scaling we have $Ric(g_\rho) \geq - \rho^2 K( r)\geq - K(r)$, and by the volume comparison theorem 
  $$Vol_{g_\rho}(B_{g_\rho}(p,1)) = \rho^{-n} Vol_g(B_g(p,\rho)) \geq c(n , K(1)) Vol_g(B_g(p,1)) .$$
From the construction in Step 1 there is a function $\tilde{f}_\rho$, such that 
  \[ d_{g_\rho} (p, x) \leq \tilde{f}_\rho(x) \leq \Lambda d_{g_\rho}(p, x),\]
  \[
  |\n_{g_\rho} \tilde{f}_\rho| \leq \Lambda, \quad |\Delta_{g_\rho} \tilde{f}_\rho| \leq \Lambda \sqrt{1+ K(4r)},  
  \]
  when $d_{g_\rho}(p, x) \geq {R}_0$. Let $f_\rho (x) = \rho \tilde{f}_\rho (x)$, then we have
  \[ d (p, x) \leq f_\rho (x) \leq \Lambda d(p, x),\]
  \[
  |\n f_\rho|(x) \leq \Lambda, \quad |\Delta f_\rho| \leq \frac{\Lambda \sqrt{1+ K(4r)}}{ \rho},  
  \]
  when $d_{g}(p, x) \geq {R}_0 \rho$. 

 For $i = 0,1,2,...$, let $\rho_i =  (2\Lambda )^{-i}$,  let $f_i = f_{\rho_i}$ as constructed above. Use $f_i$ to define cut-off functions $0\leq \phi_i \leq 1$ as in section \ref{construction of cut-off function}, let $\phi_i = 1$ on $B_{R_0 \rho_{i-1}}$, $\phi_i = 0$ outside of $B_{R_0 \rho_{i-2}}$, and 
  \[
  |\n \phi_i | \leq \frac{C(n)}{R_0 \rho_i}, \quad |\Delta \phi_i| \leq \frac{C(n)\sqrt{1+K(4R_0\rho_{i-2})}}{R_0 \rho_i ^2}.  
  \]  
  Define $\psi_0 = 1- \phi_0$, $\psi_i = \phi_i - \phi_{i+1}$, then $\psi_i \geq 0$ is supported on $B_{R_0 \rho_{i-2}} \backslash B_{R_0 \rho_{i}}$. 
  Define $f = \sum \psi_{i} f_i$, we can check similarly as in Step 1 that
  \[
  d(p,x) \leq f(x) \leq \Lambda d(p,x),  
  \]
  \[
  |\n f|(x) \leq C(n), \quad |\Delta f|(x) \leq \frac{C(n)R_0 \sqrt{1+ K(C(n)R_0)}}{d(p,x)} ,  
  \]
  for any $x \in B(p, (2\Lambda)^2 R_0) \backslash \set{p}$. Note that $f = f_0$ on $M \backslash B(p, (2\Lambda)^2 R_0)$, hence the lemma is proved after redefining $\Lambda(n)$ and $R_0$ by multiplying a suitable dimensional constant.

\end{proof}

On a manifold with nonnegative Ricci curvature, we can construct a distance-like function, whose Laplacian has alomst linear decay.
\begin{lem}\label{global distance-like function on Ricci nonnegative manifolds}
  Let $M$ be a complete noncompact Riemannian manifold with $Ric(x) \geq 0$ for all $x\in M$, let $r(x)$ be the distance from $x$ to a fixed point $p$. There exists a function $f$ satisfying
\[
r(x) \leq f(x) \leq \Lambda r(x), \]
\[ |\n f(x)| \leq \Lambda, \quad |\Delta f (x)| \leq \frac{C(n)(1- \ln \nu_p(r(x))}{r(x)},
\]
for $x \in M\backslash \set{p}$, where $\nu_p(r) = \frac{Vol B(p,r)}{\omega_n  r^n}$, $\omega_n$ is the volume of the unit ball in the Euclidean space, and the constant $\Lambda$ depends only on the dimension $n$. 
\end{lem}
\begin{proof}
  Note that $Ric \geq 0$ is invariant under scaling of the metric, hence for any $\rho> 0$, we can scale the metric $g$ and define $g_\rho = \rho^{-2}g$, which still has nonnegative Ricci curvature. In the following we will use $\n_{g_\rho}$ to denote the gradient w.r.t. $g_\rho$, and omit the subscript when the metric is $g$.
  
  \textbf{Step 1:} This step is similar to Step 2 in the proof of Lemma \ref{global distance-like function}, we define $f$ on $B(p, R_0)$.
  
  For any $0< \rho < 1$, $g_\rho$ is a scale-up of the metric $g$, $Ric(g_\rho) \geq 0$, and 
  $$Vol_{g_\rho}(B_{g_\rho}(p,1)) \geq Vol_g(B_g(p,1)) =: v_0$$
   by the volume comparison theorem. By Lemma \ref{distance-like function lemma 2} there is $R_0(n, v_0)$ and a function $\tilde{f}_\rho$, such that 
  \[ d_{g_\rho} (p, x) \leq \tilde{f}_\rho(x) \leq \Lambda d_{g_\rho}(p, x),\]
  \[
  |\n_{g_\rho} \tilde{f}_\rho| \leq \Lambda, \quad |\Delta_{g_\rho} \tilde{f}_\rho| \leq \Lambda,  
  \]
  when $d_{g_\rho}(p, x) \geq R_0$. Let $f_\rho (x) = \rho \tilde{f}_\rho (x)$, then we have
  \[ d (p, x) \leq f_\rho (x) \leq \Lambda d(p, x),\]
  \[
  |\n f_\rho|(x) \leq \Lambda, \quad |\Delta f_\rho| \leq \frac{\Lambda}{ \rho},  
  \]
  when $d_{g}(p, x) \geq R_0 \rho$. 

  Then the construction is similar as in Lemma \ref{global distance-like function}. For $i = 0,1,2,...$, let $\rho_i = (2\Lambda )^{-i}$,  let $f_i = f_{\rho_i}$ as constructed above. Use $f_i$ to define cut-off functions $0\leq \phi_i \leq 1$ as in section \ref{construction of cut-off function}, let $\phi_i = 1$ on $B_{R_0 \rho_{i-1}}$, $\phi_i = 0$ outside of $B_{R_0 \rho_{i-2}}$, and 
  \[
  |\n \phi_i | \leq \frac{C(n)}{R_0 \rho_i}, \quad |\Delta \phi_i| \leq \frac{C(n)}{ R_0 \rho_i ^2}.  
  \]  
  Define $\psi_0 = 1- \phi_0$, $\psi_i = \phi_i - \phi_{i+1}$, then $\psi_i \geq 0$ is supported on $B_{R_0 \rho_{i-2}} \backslash B_{R_0 \rho_{i}}$. 
  Define $f = \sum \psi_i f_i$, we can check similarly as in the proof of Lemma \ref{global distance-like function} that
  \[
  d(p,x) \leq f(x) \leq \Lambda d(p,x),  
  \]
  \[
  |\n f|(x) \leq C(n), \quad |\Delta f|(x) \leq \frac{C(n)R_0}{d(p,x)} = \frac{C(n)(1 + |\ln v_0|)}{d(p,x)},  
  \]
  for any $x \in B(p, (2\Lambda)^2 R_0) \backslash \set{p}$, where we have used the explicit formula for $R_0$ in Lemma \ref{distance-like function lemma 2}, and $f = f_0$ on $M \backslash B(p, (2\Lambda)^2R_0)$. In particular 
  \[ |\n f| \leq C(n)\quad and \quad |\Delta f| \leq C(n)(1 + |\ln v_0|)\] outside of $B(p,1)$, rather than outside of a ball with radius depending on the volume ratio, this fact is important for the next step, where we scale down the metric and the volume of the rescaled unit balls may degenerate. 

  \textbf{Step 2:} 
  Now for any $\rho > 1$, $g_\rho$ is a scale-down of $g$. Note that the volume of the unit ball w.r.t. $g_\rho$ is 
  \[v_0(g_\rho) = vol_{g_\rho} B_{g_\rho}(p, 1) = \frac{vol_{g} B_g(p, \rho)}{\rho^n}  = \omega_n \nu_p(\rho),\]
  where $\nu_p(\rho)$ is defined in (\ref{definition of the volume ratio}).
    By step 1, there is a distance-like function $\hat{f}_\rho$ w.r.t $g_\rho$, s.t.
  \[
  d_{g_\rho}(p,x) \leq \hat{f}_\rho(x) \leq \Lambda d_{g_\rho}(p,x),  
  \]
  \[
  |\n_{\rho} \hat{f}_\rho|(x) \leq C(n), \quad |\Delta_{g_\rho} \hat{f}_\rho|(x) \leq C(n)(1 - \ln \nu_p(\rho)),  
  \]
  for any $x \in M \backslash B_{g_\rho}(p,1)$. Let $f_\rho = \rho \hat{f}_\rho$, then
  \[
  d(p,x) \leq f_\rho(x) \leq \Lambda d(p,x),  
  \]
  \[
  |\n f_\rho|(x) \leq C(n), \quad |\Delta f_\rho|(x) \leq \frac{C(n)(1 - \ln \nu_p(\rho))}{\rho},  
  \]
  for $x \in M \backslash B(p, \rho)$.
  
  Then take $\rho_i= (2\Lambda )^i $, $i = 0, 1, 2, ...$, define $f_i = f_{\rho_i}$, then we can glue up $f_i$ by the same method as in Lemma \ref{global distance-like function} to construct a distance-like function on $M$ satisfying the requirements of the Lemma. 

\end{proof}

\begin{rem}\label{rem: local distance-like function nonnegative Ricci}
  Although we assumed $Ric \geq 0$ on the entire $M$, it easy to see that the construction of Lemma \ref{global distance-like function on Ricci nonnegative manifolds} works if $Ric \geq 0$ on a geodesic ball with radius $> R_0$, which depends on $n$ and the volume of a unit ball centered at the base point. 
\end{rem}

\begin{rem}
In the case of $Ric \geq 0$ and maximal volume growth, there is an alternative choice of distance-like function. One can simply take $G^{\frac{1}{n-2}}$ as a distance-like function as in \cite{CM1997}, where $G$ is the positive Green's function, it follows from Cheng-Yau's gradient estimate that this function has bounded gradient and linearly decaying Laplacian. 
\end{rem}

\begin{rem} The functions $r(x)$, $f(x)$ and $K(r(x))$ defined above, and the constant $R_0$ may depend on a base point $p$, which is suppressed when there is no confusion. We will write $r_p(x)$, $f_p(x)$, $K_p(r_p(x))$ and $R_0(p)$ when we need to specify the base point. 
\end{rem}

\section{$L^2$ estimates and a mean value inequality for the biharmonic heat equation}\label{section: l2 estimates and mean value inequality}
The main purpuse of this section is to develop a mean value inequality for the biharmonic heat equation. Let's first fix some notations. 
Let $f$ be the distance-like function based at a point $p$ as defined in Lemma \ref{global distance-like function} (or Lemma \ref{global distance-like function on Ricci nonnegative manifolds} if the Ricci curvature is nonnegative) and denote 
$$D_r = \set{f< r}, $$ 
$r > 0$. Let $R_0$ be the same constant as in Lemma \ref{global distance-like function}, WLOG we can take $R_0 \geq 1$.

\subsection{$L^2$ estimates for the biharmonic heat equation. }
\begin{lem}\label{lem: l2 estimates for poly-Laplacian}
Suppose $Ric \geq -K$ on $D_R$, let $u$ be a solution of the biharmonic heat equation on $D_R\times [0, T]$, then for $R>0$, $0< t \leq T$, $m = 0, 1, 2, ...$ we have
\[
\int_{D_{R/2}} |\Delta^m u|^2 (x, t)\leq \frac{C(n,m)\left( \max \set{R_0, R}^2 \max \set{1,K} R^{-4}T + 1\right)^{m+1}}{t^{m+1}} \int_0^t \int_{D_R} u^2.
\]
In the special case $K = 0$, we have 
\[
\int_{D_{R/2}} |\Delta^m u|^2 (x, t)\leq \frac{C(n,m)\left((1- \ln \nu_p(R))R^{-4}T + 1\right)^{m+1}}{t^{m+1}} \int_0^t \int_{D_R} u^2,
\]
for $R>0$ and $0< t \leq T$, $\nu_p(r)$ is the volume ratio of the geodesic ball with radius $r$ centered at the base point of $f$.
\end{lem}
\begin{proof}
Let $\phi$ be a cut-off function supported on $D_l$, with $\phi = 1$ on $D_{s}$. Let's first consider the case $R_0<  s < l \leq R$, then using the method in section \ref{construction of cut-off function}, we can make $\phi$ satisfy
\begin{equation}\label{eqn: delta of phi in proof of l2 estimate}
|\n \phi| \leq \frac{C(n,k)}{l-s} \phi^{1-1/k}, \quad |\Delta \phi| \leq \max \set { \frac{C(n,k)\sqrt{1+K}}{l-s} \phi^{1-2/k}, \quad \frac{C(n,k)}{(l-s)^2} \phi^{1-2/k} },
\end{equation}
where $k$ is taken to be $\geq 4$. 

We start by taking the time derivative of the following $L^2$-integral, where $m\geq 0$ is any integer, then apply integration by parts and Cauchy-Schwarz inequality,
\[
\begin{split}
&\frac{d}{dt} \int |\Delta^m u|^2 \phi^2 \\
= & - \int 2 \Delta^m u \Delta^{m+2} u \phi^2 \\
= & \int 2 \langle \n \Delta^m u, \n \Delta^{m+1} u \rangle\phi^2 + 4 \phi \Delta^m u \langle \n \Delta^{m+1} u, \n \phi \rangle  \\
= & - \int 2 |\Delta^{m+1} u|^2 \phi^2 + 8 \phi \Delta^{m+1} u \langle \n \Delta^m u, \n \phi \rangle + 4 \Delta^m u \Delta^{m+1} u \left( \phi \Delta \phi + |\n \phi|^2 \right) \\
\leq & - (2-3\epsilon)\int |\Delta^{m+1} u|^2 \phi^2 + \epsilon^{-1}\int 16 |\n \Delta^m u|^2 |\n \phi|^2 + 4 |\Delta^m u|^2 \left( |\Delta \phi|^2 + |\n \phi|^4 \phi^{-2} \right).
\end{split}
\]
By the choice of $\phi$, we can apply integration by parts and Cauchy-Schwarz inequality to get,  
\[
\begin{split}
& \int |\n \Delta^m u|^2 |\n \phi|^2 \\
\leq &  \frac{C(n,k)}{(l-s)^2} \int |\n \Delta^m u|^2 \phi^{2-2/k}\\
= & \frac{C(n,k)}{(l-s)^2} \int - \Delta^m u \Delta^{m+1} u \phi^{2-2/k}- (2-2/k) \phi^{1-2/k} \Delta^m u \langle \n \Delta^m u, \n \phi \rangle \\
\leq & \frac{1}{2}\epsilon^2 \int |\Delta^{m+1} u|^2 \phi^2 + \frac{C(n,k)}{(l-s)^4} (1+\epsilon^{-2}) \int |\Delta^m u|^2 \phi^{2-4/k}  + \frac{1}{2} \int |\n \Delta^m u|^2 |\n \phi|^2. 
\end{split}
\]
Take $\epsilon = \frac{1}{4}$, then the first term in the last line can be absorbed, and we have 
\[
\begin{split}
\frac{d}{dt} \int |\Delta^m u|^2 \phi^2 
\leq -\int |\Delta^{m+1} u|^2 \phi^2 + C(n,k) \int |\Delta^m u|^2 \left( (l-s)^{-4}\phi^{2-4/k} + |\Delta \phi|^2 \right).
\end{split}
\]
By the bound of $\Delta \phi$, we have
\begin{equation}\label{eqn: time derivative of the l2 integral prototype}
\frac{d}{dt} \int |\Delta^m u|^2 \phi^2 
\leq -\int |\Delta^{m+1} u|^2 \phi^2 + C_1 \int |\Delta^m u|^2  \phi^{2-4/k} 
\end{equation}
where
\[
C_1 = \frac{C(n,k) (1+K)}{(l-s)^2} + \frac{C(n,k)}{(l-s)^4}.
\]

Similarly, for $i = 0, 1, 2,..., m$, let $\phi_i = 1 $ on $D_{(1-\frac{i+1}{2m+2})R}$ and $\phi_i =0$ on $M \backslash D_{(1-\frac{i}{2m+2})R}$, (\ref{eqn: time derivative of the l2 integral prototype}) becomes
\begin{equation}\label{eqn: time derivative of the l2 integral}
\frac{d}{dt} \int |\Delta^i u|^2 \phi_i^2 
\leq -\int |\Delta^{i+1} u|^2 \phi_i^2 + C_2 \int |\Delta^i u|^2  \phi_i^{2-4/k} ,
\end{equation}
where 
\[
C_2 = \frac{C(n,m, k)}{R^2} \left(1+K + \frac{1}{R^2} \right) \leq C(n,m,k) \frac{\max \set{1, K} }{R^2} ,
\]
where we have used the assumption that $R > R_0 \geq 1$.
Now define
\[
F_m(t) = \sum_{i=0}^m a_i t^{i+1} \int |\Delta^i u|^2 \phi_i^2, \quad t \in [0,T].
\]
By (\ref{eqn: time derivative of the l2 integral}),
\[\begin{split}
& \frac{d}{dt} F_m(t) \\
\leq &\sum_{i=1}^m \left( C_2 a_i t^{i+1} \int |\Delta^i u|^2 \phi_i^{2-4/k} + (i+1) a_i t^{i} \int |\Delta^i u|^2 \phi_i^2 - a_{i-1}t^{i} \int |\Delta^i u |^2 \phi_{i-1}^2 \right) \\
& + a_0 \int u^2 \phi_0^2 + C_2 a_0 t \int u^2 \phi_0^{2-4/k} - a_m t^{m+1} \int |\Delta^{m+1} u|^2 \phi_m^2.
\end{split}
\]
Take
\begin{equation}\label{definition of a_m}
a_m = 1, \quad a_{i-1} = (C_2 T + i + 1) a_{i}, \quad for \quad i = m, m-1, m-2, ...,1.
\end{equation}
Then
\[
  \frac{d}{dt} F_m(t) \leq a_0 \int u^2 \phi_0^2 + C_2 a_0 t \int u^2 \phi_0^{2-4/k}.
\]
Since $F_m(0) = 0$ and $\phi_0$ is supported on $D_r$, we have
\[
F_m(t) \leq (1+C_2T) a_0 \int_0^t \int_{D_R}|u(x,s)|^2 .
\]
Fix $k =4$, the constants can be estimated using the inductive relation (\ref{definition of a_m}) and we finish the proof. 

When $0 < s < l <R \leq  2 R_0$, the distance-like function $f$ satisfies
\[
|\Delta f | \leq \frac{\Lambda(n) R_0 \sqrt{1+K}}{d(p,x)} \leq \frac{C(n) R_0 \sqrt{1+K}}{f(x)} ,
\]
hence we can verify that the cut-off function $\phi$ satisfies
\[
|\Delta \phi| \leq \max \set {\frac{C(n,k)R_0 \sqrt{1+K}}{s(l-s)} \phi^{1-2/k}, \quad \frac{C(n,k)}{(l-s)^2} \phi^{1-2/k} }.  
\]
Then the constant $C_2$ in (\ref{eqn: time derivative of the l2 integral}) can be taken to be
\[
C_2 =  \frac{C(n, m,k)R_0^2 \max \set{1,K}}{R^4},  
\]
the rest of the proof is the same as before.

In the special case that $K = 0$, by Lemma \ref{global distance-like function on Ricci nonnegative manifolds} and Remark \ref{rem: local distance-like function nonnegative Ricci}, we have a distance-like function $f$ satisfying
\[ |\Delta f(x)| \leq \frac{C(n)(1- \ln \nu_p(d(p,x)))}{d(p,x)} \leq \frac{ \Lambda(n)C(n)(1- \ln \nu_p(f(x)))}{f(x)},\]
hence the cut-off function $\phi$ supported on $D_l$ and equals $1$ on $D_s$ can be taken to satisfy
\[
|\Delta \phi| \leq \max \set { \frac{C(n)(1- \ln \nu_p(l) )}{s(l-s)} \phi^{1-2/k} , \quad \frac{C(n,k)}{(l-s)^2} \phi^{1 - 2/k} }  
\]
instead of (\ref{eqn: delta of phi in proof of l2 estimate}). Then the constant $C_2$ in (\ref{eqn: time derivative of the l2 integral}) can be taken to be
\[
C_2 =  \frac{C(n, m,k)(1- \ln \nu_p(R))}{R^4}.  
\]

\end{proof}

\subsection{A mean value inequality for the biharmonic heat equation}
\subsubsection{Multiplicative Sobolev inequality.}
The following Sobolev inequality is well-known \cite{Sal1992}. Here $B_r$ denotes a geodesic ball with radius $r$, and $V(r) = Vol(B_r)$.
\begin{lem}\label{lem: Sobolev inequality}
Suppose $Ric \geq -K$ on $B_r$, then for any $u \in W_0^{1,2}(B_r)$, 
\[
\left( \int_{B_r} u^{\frac{2\mu}{\mu-2}}\right)^{\frac{\mu - 2}{\mu}} \leq \frac{C(n,\mu) e^{\sqrt{K}r}}{V(r)^{\frac{2}{\mu}}} \int_{B_r} r^2 |\n u|^2 + u^2,
\]
where $\mu \geq n$ when $n \geq 3$ and $\mu > 2$ when $n = 2$. 
\end{lem}
Since $u$ has compact support, by integration by parts, we have
\[
\int|\n u|^2 = - \int u \Delta u \leq \frac{1}{2}\int r^2|\Delta u|^2 + r^{-2}u^2,
\]
hence we have
\begin{lem}
Suppose $Ric \geq -K$ on $B_r$, then for any $u \in W_0^{2,2}(B_r)$, 
\[
\left( \int_{B_r} u^{\frac{2\mu}{\mu-2}}\right)^{\frac{\mu - 2}{\mu}} \leq \frac{C(n,\mu) e^{\sqrt{K}r}}{V(r)^{\frac{2}{\mu}}} \int_{B_r} r^4 |\Delta u|^2 + u^2,
\]
where $\mu \geq n$ when $n \geq 3$ and $\mu > 2$ when $n = 2$. 
\end{lem}
This Sobolev inequality can be used to prove a multiplicative version. Here we use the method of \cite{LSU1968}, see also the appendix of \cite{KS2002}. Denote  
\[
\|u\|_{r, p} = \left( \int_{B_r} |u|^p\right)^\frac{1}{p}.
\]
Let $\tau \geq 0$, by the Sobolev inequality,
\[
\left( \int_{B_r} |u|^{(1+\tau)\frac{2\mu}{\mu-2}}\right)^{\frac{\mu - 2}{\mu}} \leq \frac{C(n,\mu) e^{\sqrt{K}r}}{V(r)^{\frac{2}{\mu}}} \int_{B_r}  r^2 (1+\tau)^2|u|^{2\tau}|\n |u||^2 + |u|^{2(1+\tau)},
\]
by integration by parts,
\[
\int |u|^{2\tau}|\n |u| |^2 =  \frac{1}{1+2\tau}\int\langle \n |u|^{1+2\tau}, \n |u|\rangle  = - \frac{1}{1+2\tau} \int |u|^{1+2\tau} \Delta |u|,
\]
note that $\Delta |u|$ is well defined when $u \neq 0$ so $|u|^{1+2\tau} |\Delta |u|| \leq |u|^{1+2\tau} |\Delta u|$ a.e. By the Holder inequality we have
\[
\int |u|^{2\tau}|\n |u| |^2 \leq \frac{1}{1+2\tau}\left( \int |u|^{(1+2\tau)q}\right)^\frac{1}{q}\left(\int |\Delta u|^p \right)^\frac{1}{p},
\]
\[
\int |u|^{2(1+\tau)} \leq \left( \int |u|^{(1+2\tau)q}\right)^\frac{1}{q}\left(\int | u|^p \right)^\frac{1}{p},
\]
where $p>1$ and $\frac{1}{p} + \frac{1}{q} = 1$. Hence
\[
\|u\|_{r, \frac{2(1+\tau)\mu}{\mu-2}}^{2(1+\tau)} \leq C_S (1+\tau)^2 \|u\|_{r, (1+2\tau)q}^{1+2\tau} \left( \frac{r^2}{1+2\tau}  \| \Delta u \|_{r, p} +  \|u\|_{r, p}  \right),
\]
where
\[
C_S = \frac{C(n,\mu) e^{\sqrt{K}r}}{V(r)^{\frac{2}{\mu}}} .
\]
Assume (by multiplying a suitable constant to $u$) that
\[
 r^2 \|  \Delta u \|_{r, p} +  \|u\|_{r, p} = 1.
\]
Then
\[
\|u\|_{r, \frac{2(1+\tau)\mu}{\mu-2}}^{} \leq C_S^{\frac{1}{2(1+\tau)}} (1+\tau)^\frac{1}{1+\tau} \|u\|_{r, (1+2\tau)q}^{\frac{1+2\tau}{2(1+\tau)}} .
\]
Take 
\[
0\leq \tau_0 < \frac{1}{2}, \quad \tau_{i+1} = \frac{\mu}{\mu - 2} \frac{1}{q} (1+ \tau_i) - \frac{1}{2}, \quad  i =0, 1 ,2,...
\]
we see that 
\[1+  2\tau_{i+1} = k 2 (1+\tau_i)  , \quad where \quad k = \frac{\mu}{(\mu -2)q},\]
and by direct calculation
\[
1 + \tau_i = k^i (1+ \tau_0) + \frac{k^i - 1}{2(k - 1)}.
\]
Then
\[
\|u\|_{r, (1+2 \tau_{i+1})q}^{} \leq C_S^{\frac{1}{2(1+\tau_i)}} (1+\tau_i)^\frac{1}{1+\tau_i} \|u\|_{r, (1+2\tau_i)q}^{\frac{1+2\tau_i}{2(1+\tau_i)}} .
\]
Iterating the above inequality yields
\[
|u|_{r, \infty} \leq C_S^{\sum \frac{1}{2(1+ \tau_i)}} e^{\sum \frac{\ln(1+ \tau_i)}{1+ \tau_i}} \|u\|_{r, (1+2\tau_0)q}^{\prod \frac{1+2\tau_i}{2(1+\tau_i)}},
\]
where we have used the fact $\frac{1+2\tau_i}{2(1+\tau_i)} < 1$ to simplify the calculation, and we impose the condition that
\[
k=\frac{\mu }{\mu - 2}\frac{1}{q} > 1 , \quad i.e. \quad p > \frac{\mu}{2},
\]
which ensures $\tau_i \to \infty$ and the convergence of the series. 
By direct calculation,
\begin{equation}\label{eqn: definition of alpha}
\alpha : = \prod_{i=1}^\infty \frac{1+2\tau_i}{2(1+\tau_i)} = \lim_{j\to \infty} \frac{(1+2\tau_0)k^j}{2(1+\tau_j)} =  \frac{1+2\tau_0}{2(1+\tau_0) + (k-1)^{-1}}  \in (0,1),
\end{equation}
\begin{equation}\label{eqn: definition of beta}
\beta : = \sum \frac{1}{2(1+ \tau_i)} \leq \frac{1}{2(1+ \tau_0)} \frac{\mu}{\mu - q (\mu - 2)} .
\end{equation}
Therefore we have proved the following multiplicative Sobolev inequality:
\begin{lem} Suppose $Ric \geq -K$ in the geodesic ball $B_r$, then for any $u \in W_0^{2,2}(B_r)$, we have
\[
\norm {u}_{r, \infty} \leq C(n, \mu , p, \tau_0) C_S^\beta  \|u\|_{r, (1+2 \tau_0) q}^{\alpha} \left( r^2 \|\Delta u\|_{r, p} + \|u\|_{r, p} \right)^{1-\alpha},
\]
where $p > \frac{\mu}{2}$, $q = \frac{p}{p-1}$, $0\leq \tau_0 < \frac{1}{2}$, $\mu$ is as in Lemma \ref{lem: Sobolev inequality}, and $\alpha, \beta$ are defined in (\ref{eqn: definition of alpha}, \ref{eqn: definition of beta}). 
\end{lem}
\subsubsection{Using averaged norm.}
Let 
\[
\mnorm{u}_{r, p} = \left(\frac{1}{Vol B_r}  \int_{B_r} |u|^p \right)^\frac{1}{p}.
\]
Then the Sobolev inequality takes the form
\[
\mnorm{u}_{r, \frac{2\mu}{\mu - 2}}^2 \leq C(n, \mu)e^{\sqrt{K}r}  \left(r^2 \mnorm{\n u}_{r,2}^2 +  \mnorm{u}_{r, 2}^2 \right).
\]
The above argument yields
\begin{lem} Suppose $Ric \geq -K$ in the geodesic ball $B_r$, then for any $u \in W_0^{2,2}(B_r)$, we have
\begin{equation}\label{multiplicative sobolev inequality with averaged norm}
\mnorm{u}_{r, \infty} \leq C(n, \mu , p, \tau_0) e^{\beta\sqrt{K}r} \mnorm{u}_{r, (1+2 \tau_0) q}^{\alpha} \left( r^2 \mnorm {\Delta u}_{r, p} +  \mnorm{u}_{r, p} \right)^{1-\alpha},
\end{equation}
where $p > \frac{\mu}{2}$, $q = \frac{p}{p-1}$, $0\leq \tau_0 < \frac{1}{2}$, $\mu$ is as in Lemma \ref{lem: Sobolev inequality}, and $\alpha, \beta$ are defined in (\ref{eqn: definition of alpha}, \ref{eqn: definition of beta}). 
\end{lem}

\subsubsection{$L^p$ estimates.} From now on we do not assume $u$ to be compactly supported. 
The Sobolev inequalities can help us to control $L^p$ norms of $u$ and $\Delta u$ in terms of $L^2$ norms of $\Delta^i u$, $i = 0, 1, 2,...,m$ for $m$ large enough. 

Let $\phi$ be a cut-off function supported on $B_r$ and $\phi = 1$ on $B_{(1-\frac{1}{m})r}$, such that $|\n \phi| \leq C(n,m) r^{-1}$. by the Sobolev inequality we have
\[
\mnorm{|u|^s \phi }_{r, \frac{2\mu}{\mu - 2}}^2 \leq C(n, \mu)e^{\sqrt{K}r}  \left(r^2 \mnorm{\n (|u|^s\phi) }_{r,2}^2 +  \mnorm{|u|^s\phi}_{r, 2}^2 \right).
\]
Use integration by parts and Cauchy -Schwarz inequality to estimate
\[
\begin{split}
\int |\n (|u|^s \phi)|^2 = & \int s^2 |u|^{2s-2} |\n |u||^2 \phi^2 + |u|^{2s}|\n \phi|^2 + 2 s \phi |u|^{2s-1} \langle \n |u|, \n \phi\rangle \\
= & - \int \frac{s^2}{2s-1} |u|^{2s-1} \Delta |u| \phi^2 + \int |u|^{2s} |\n \phi|^2 \\
& + \int \left( \frac{s^2}{2s-1} +s \right) |u|^{2s-1} \langle \n |u|, \n \phi^2\rangle\\
= & - \int \frac{s^2}{2s-1} |u|^{2s-1} \Delta |u| \phi^2 + \int |u|^{2s} |\n \phi|^2 \\
& + \int \left( \frac{s}{2s-1} +1 \right) |u|^{s} \langle \n (|u|^s \phi) - |u|^s \n \phi, 2 \n \phi \rangle\\
\leq  & - \int \frac{s^2}{2s-1} |u|^{2s-1} \Delta |u| \phi^2  + \int 4 \left( \frac{s}{2s-1} +1 \right)^2 |u|^{2s} |\n \phi|^2 \\
& + \frac{1}{2} \int |\n (|u|^s \phi)|^2.
\end{split}
\]
Absorb the last term by the LHS, then by Holder inequality and the bounds on $|\n \phi|$, 
\[
  \begin{split}
\int |\n (|u|^s \phi)|^2 \leq & C(n,s) \left(\int_{spt(\phi)}  |u|^{2s} \right)^{\frac{2s-1}{2s}} \left( \int_{spt(\phi)} |\Delta u|^{2s}\right)^\frac{1}{2s} \\
& + \frac{C(n,m,s)}{r^2} \int_{spt(\phi)} |u|^{2s} ,
  \end{split}
\]
suppose $r>1$, by Young\rq{}s inequality we can estimate the gradient term:
\[
\int |\n (|u|^s \phi)|^2 \leq C(n,s) \int_{spt(\phi)} r^{2(2s-1)} |\Delta u|^{2s}+ C(n,m,s) r^{-2} \int_{spt(\phi)} |u|^{2s}.
\]
By this estimate and the volume comparison theorem, we have the following inequality in averaged norm,
\[
\mnorm{u}_{(1-\frac{1}{m})r, \frac{2s\mu}{\mu - 2}}^2 \leq C(n, \mu,m)^{1/s}e^{\sqrt{K}r/s}  \left( \mnorm{ r^2 \Delta u }_{r,2s}^2 +  \mnorm{u}_{r, 2s}^2 \right).
\]
Now let 
\[
s_i = \left(\frac{\mu}{\mu-2}\right)^i , \quad  r_i = \left( 1- \frac{i}{4m} \right) r, \quad i = 0,1,2,..., m.
\]
By choosing the cut-off function $\phi = 1$ on $D_{r_{i+1}}$ and supported on $D_{r_{i}}$, and set $s= s_{i}$, the above argument yields that for $i \geq 1$,
\[
\mnorm{u  }_{r_{i}, 2s_{i}}^{} \leq C(n, \mu,m)^{\frac{1}{s_{i-1}}}e^{\sqrt{K}r\frac{1}{2s_{i-1}}} \left(  \mnorm{r^2\Delta u}_{r_{i-1}, 2s_{i-1}} + \mnorm{u}_{r_{i-1}, 2s_{i-1}}\right).
\]
Iterating this inequality yields
\[
\mnorm{u  }_{r_{i}, 2s_{i}}^{} \leq C(n, \mu,m)^{\sum_{j=0}^{i-1}\frac{1}{s_{j}}}e^{\sqrt{K}r\sum_{j=0}^{i-1}\frac{1}{2s_{j}} }  \sum_{j=0}^{i} \frac{i !}{j! (i-j)!} \mnorm{(r^2\Delta)^j u}_{r, 2}.
\]
Since $\sum_{j=0}^\infty \frac{1}{s_j} = \frac{\mu}{2}$, we have 
\[
\mnorm{u  }_{r_{i}, 2s_{i}}^{} \leq C(n, \mu,m)^{\frac{\mu}{2}}e^{\sqrt{K}r \frac{\mu}{4} }  \sum_{j=0}^{i} \frac{i !}{j! (i-j)!} \mnorm{(r^2\Delta)^j u}_{r, 2}.
\]
Similarly, using $\Delta u$ instead of $u$ yields
\[
\mnorm{ r^2 \Delta u  }_{r_{i}, 2s_{i}}^{} \leq C(n, \mu,m)^{\frac{\mu}{2}}e^{\sqrt{K}r \frac{\mu}{4} }  \sum_{j=0}^{i} \frac{i !}{j! (i-j)!} \mnorm{(r^2\Delta)^{j+1} u}_{r, 2}.
\]
Since $s_i \to \infty$, for any $p>0$, we can choose $m$ large enough so that $s_m > p$. Therefore, by the multiplicative Sobolev inequality (\ref{multiplicative sobolev inequality with averaged norm}), where we take 
$$\tau_0 = \frac{1}{2} - \frac{1}{p},$$ 
so that $(1+2\tau_0)q = 2$, we have:
\begin{lem}\label{lem: l2 multiplicative sobolev inequality} For any function $u$ with $\norm{\Delta^i u}_{r,2} < \infty$ for $i = 0, 1,2,...,m+1$, 
\begin{equation}\label{L2 multiplicative sobolev inequality}
\mnorm{u}_{\frac{r}{2}, \infty} \leq C(n, \mu , p) e^{C(n, \mu, p)\sqrt{K}r} \mnorm{u}_{r,2}^{\alpha} \left(
\sum_{j=0}^{m+1} \frac{(m+1) !}{j! (m+1-j)!} \mnorm{(r^2\Delta)^{j} u}_{r, 2}  \right)^{1-\alpha} ,
\end{equation}
where $m \geq \log_{\frac{\mu}{\mu-2}} \frac{p}{2} > m-1$, $p > \frac{\mu}{2}$, $\mu$ is as in Lemma \ref{lem: Sobolev inequality} and $\alpha$ is defined in (\ref{eqn: definition of alpha}). 
\end{lem}

\subsubsection{A mean value inequality for the biharmonic heat equation.}
Now let $u$ be a solution of the biharmonic heat equation. 
By Lemma \ref{lem: l2 estimates for poly-Laplacian}, where we take $R = 2\Lambda r$, and note that 
\[
B_r \subset D_{\Lambda r}, 
\]we have 
\[
\int_{B_r} | (r^2 \Delta)^i u|^{2}(x,t) \leq C_{3,i} \frac{1}{r^4} \int_0^t \int_{D_{2 \Lambda r}} |u|^2, \quad i= 0,1,...,m+1,
\]
where 
\[
C_{3,i} = C(n,i)\left(\max \set{R_0, r}^2 \max \set{1,K} + \frac{r^4}{t}\right)^{i+1} \quad when \quad K > 0;
\]
\[
C_{3,i} = C(n,i,v_0)\left( 1 - \ln \nu(r)+ \frac{r^4}{t}\right)^{i+1} \quad when \quad K =  0,
\]
and $C_{3,i}$ can be chosen as a increasing sequence in $i = 0, 1, 2, ...$.

Hence by (\ref{L2 multiplicative sobolev inequality}) and the volume comparison theorem, 
\begin{equation}
\begin{split}
& |u(x,t)|_{r/2, \infty} \\
 \leq & C(n,\mu,p)e^{C(n,\mu,p)\sqrt{K}r}  \frac{\left(\max \set{1,K}\max \set{R_0,r}^2+\frac{r^2}{\sqrt{t}}\right)^{(m+1)(1-\alpha)+1}}{r^2 \sqrt{V(D_{2\Lambda r})}} \left( \int_0^t \int_{D_{2\Lambda r}} u^2\right)^\frac{1}{2},
\end{split}
\end{equation}
when $K > 0$; 
and 
\begin{equation}
\begin{split}
& |u(x,t)|_{r/2, \infty} \\
 \leq & C(n,\mu,p)  \frac{\left( \sqrt{1- \ln \nu(r)}+\frac{r^2}{\sqrt{t}}\right)^{(m+1)(1-\alpha)+1}}{r^2\sqrt{V(D_{2\Lambda r})}} \left( \int_0^t \int_{D_{2\Lambda r}} u^2\right)^\frac{1}{2},
\end{split}
\end{equation}
when $K = 0$, 
where $m \geq \log_{\frac{\mu}{\mu-2}} \frac{p}{2} > m-1$, and $p > \frac{\mu}{2}$.

Now, choose integer $m$ such that $\frac{\mu}{2} < \left(\frac{\mu}{\mu-2}\right)^m \leq \frac{\mu}{2} + 1$, and take $p = \left(\frac{\mu}{\mu-2}\right)^m $, then we have proved a mean value inequality for the biharmonic heat equation:
\begin{lem}[Mean value inequality]\label{lem: mean value inequality}
 Let $u$ be a solution of the biharmonic heat equation on $D_{2\Lambda r}$, where $Ric \geq -K$ then we have the mean value inequality
\begin{equation}
|u(x,t)|_{r/2, \infty} \leq \frac{\Gamma}{r^2\sqrt{V(D_{2\Lambda r})}} \left( \int_0^t \int_{D_{2\Lambda r}} u^2\right)^\frac{1}{2},
\end{equation}
with
\begin{equation}\label{eqn: definition of Gamma}
\Gamma = C(n,\mu)\left(\max \set{R_0, r}^2 \max \set{1,K}+\frac{r^2}{\sqrt{t}}\right)^{c(\mu)} e^{C(n,\mu)\sqrt{K}r} ,
\end{equation}
where we can take $\mu = n$ when $n \geq 3$ and $\mu > 2$ when $n=2$. Moreover, in the case $K=0$ we can take 
\[
  \Gamma = C(n,\mu)\left(\sqrt{1- \ln \nu(r)}+\frac{r^2}{\sqrt{t}}\right)^{c(\mu)},
\] 
for $r> 0$ where $\nu(r)$ is the volume ratio of the $r$-ball centered at the base point of $f$.
\end{lem}

\section{Integral decay estimates for $L^2$ solutions.}
In this section we first study the growth of a weighted $L^2$ integral of the solution of the biharmonic heat equation, we can show that this integral is monotone in time if the weight functions are properly chosen and if the solution has certain growth rate. This generalize the corresponding formula for the heat equation to fourth order. 
As a direct application, we assume $u$ is an $L^2(M)$ solution of the biharmonic heat equation, and assume that the Ricci lower bound has at most quadratic growth, then we obtain exponential decay estimates for $u$ in $L^2$ form. 

\subsection{Growth of a weighted $L^2$ integral}
The following is a main technical lemma for the rest of this article, it will be applied with different choices of weight functions $\xi$ and $G$. 
\begin{lem}\label{time derivative of weighted l2 integral}
Suppose $Ric(x) \geq - K(r(x))$ for some nondecreasing function $K(r)$. Let $R > \Lambda R_0$ and $0< \rho$.  For any $k \geq 4$, let $\phi$ be a cut-off function constructed as in section \ref{construction of cut-off function},  satisfying $\phi = 1$ on $D_R$ and $\phi = 0$ on $M\backslash D_{R+\rho}$. And let $\xi = \xi(f(x),t)$ be a $C^1$ function whose Laplacian exists a.e. and let $G(x,t)$ be a Lipschitz function satisfying
\[
|\n \xi|^2(x, t) \leq G(x, t),
\]
then for any solution $u$ of the biharmonic heat equation, we have 
\begin{equation}
\begin{split}
\partial_t \int u^2 e^\xi \phi^2 \leq  & \int u^2 e^\xi \phi^2 \left(\partial_t \xi + C G^2 + C |\n G|^2 G^{-1} + C |\Delta \xi|^2 \right) \\
& + \frac{C(n,k)(1+K(R+ \rho))}{\rho^2} \int_{D_{R+\rho} \backslash D_{R}} u^2 e^{\xi} \phi^{2-4/k},
\end{split}
\end{equation}
where $C$ is some universal constant. 
\end{lem}

\begin{proof}
\[
\begin{split}
& \frac{d}{dt}\int u^2 e^\xi \phi^2 \\
 = & \int - 2 u \Delta ^2 u e^\xi \phi^2 + u^2 e^\xi \phi^2 \partial_t \xi \\
= & \int 2 \langle \n \Delta u, \n u\rangle e^\xi \phi^2 +2 \langle \n \Delta u, \n \xi\rangle u e^\xi \phi^2 + 4 \langle \n \Delta u, \n \phi \rangle u e^\xi \phi  + \int u^2 e^\xi \phi^2 \partial_t \xi \\
=& \int - 2 (\Delta u)^2 e^\xi \phi^2 - 2 \langle \n u, \n \xi \rangle \Delta u e^{\xi} \phi^2 - 4 \langle \n u, \n \phi \rangle \Delta u e^\xi \phi \\
& + \int - 2 u \Delta u \Delta \xi e^\xi \phi^2 - 2\langle \n u, \n \xi \rangle \Delta u e^{\xi} \phi^2 - 2|\n \xi|^2 u \Delta u e^{\xi} \phi^2 -   4\langle \n \xi, \n \phi\rangle u \Delta u e^{\xi} \phi \\
&+ \int -4 \langle \n u, \n \phi \rangle \Delta u e^{\xi} \phi - 4\langle \n \phi, \n \xi \rangle u \Delta u e^\xi \phi   - 4 |\n \phi|^2 u \Delta u e^\xi \\
&- 4 \int \Delta u \Delta \phi u e^\xi \phi + \int u^2 e^\xi \phi^2 \partial_t \xi,
\end{split}
\]
where we have applied integration by parts twice. 

Apply the Cauchy-Schwarz inequality to each term containing $\Delta u$, we get
\[
\begin{split}
& \frac{d}{dt}\int u^2 e^\xi \phi^2 - \int u^2 e^\xi \phi^2 \partial_t \xi \\
\leq & \int - 2 (\Delta u)^2 e^\xi \phi^2 + 10\epsilon \int (\Delta u)^2 e^\xi \phi^2 + 2\epsilon^{-1} \int |\n u|^2 |\n \xi|^2 e^{\xi} \phi^2 
+ 8\epsilon^{-1}\int |\n u|^2 e^\xi |\n \phi|^2 \\
& + \epsilon^{-1}\int u^2 e^{\xi } \left( |\n \xi|^4  \phi^2 + |\Delta \xi|^2 + 8  |\n \xi|^2 |\n \phi|^2 +  4|\n \phi|^4 \phi^{-2}   +  4|\Delta \phi|^2\right)
\end{split}
\]

Choose $\epsilon = \frac{1}{10}$, and apply the Cauchy-Schwarz inequality to
\[2 |\n \xi|^2 |\n \phi|^2 \leq |\n \xi|^4 + |\n \phi|^4,\]
 we have 
\begin{equation}\label{estimate of the time derivative by I and II}
\begin{split}
 &\frac{d}{dt}\int u^2 e^\xi \phi^2 - \int u^2 e^\xi \phi^2 \partial_t \xi \\
\leq & - \int  (\Delta u)^2 e^\xi \phi^2  + 80\int u^2 e^{\xi} \left( |\n \xi|^4  \phi^2  + |\Delta \xi|^2 \phi^2+  |\n \phi|^4 \phi^{-2}  +  |\Delta \phi|^2 \right) \\
& + I + II,  
\end{split}
\end{equation}
where 
\[
I =  20 \int |\n u|^2 |\n \xi|^2 e^{\xi} \phi^2, \quad  
II = 80 \int |\n u|^2 e^\xi |\n \phi|^2. \\
\]

Let's first estimate $I$. By the assumption we have
\[
I =  20 \int |\n u|^2 |\n \xi|^2 e^\xi \phi^2 \leq 20 \int |\n u|^2 G e^\xi \phi^2.
\]
Integration by parts yields
\[
\begin{split}
I \leq& 20 \int |\n u|^2 G e^\xi \phi^2  \\
= &  20 \int - u \Delta u G e^\xi \phi^2 - u \langle \n u, \n G \rangle e^\xi \phi^2 - u \langle \n u, \n \xi \rangle G e^\xi \phi^2  - 2 u \langle\n u, \n \phi \rangle G e^\xi \phi  .
\end{split}
\]
Using Cauchy-Schwarz inequality for each term, we have 
\[
\begin{split}
I \leq& 20 \int |\n u|^2 G e^\xi \phi^2 \\
\leq & \frac{1}{8} \int |\Delta u|^2 e^ \xi \phi^2 + \int 800  u^2 e^\xi G^2 \phi^2  +  |\n u|^2 G e^\xi \phi^2 + 400 u^2 |\n G |^2 G^{-1} e^\xi \phi^2  \\
& + \int 9  |\n u|^2 G e^\xi \phi^2  + \int 25  u^2 G^2 e^\xi \phi^2   + 80  u^2 G e^\xi |\n \phi|^2. 
\end{split}
\]
After absorbing the terms containing $\int |\n u|^2 G e^\xi \phi^2$ by the LHS, we have
\begin{equation}\label{estimate of I}
\begin{split}
I \leq & \frac{1}{4} \int |\Delta u|^2 e^ \xi \phi^2  + C \int u^2 e^\xi \left( G^2 \phi^2 + |\n G |^2 G^{-1}  \phi^2  
 +  G  |\n \phi|^2 \right),
\end{split}
\end{equation}
where $C$ is some universal constant. 

Next we estimate $II$. Recall that $\xi(x, t) = \xi(f(x), t)$ is radial in the gradient direction of $f$, hence we can define
\[
\xi_R(x, t) = \begin{cases}
\xi(R, t), & x \in D_R; \\
\xi(x, t), & x \in M \backslash D_R.
\end{cases}
\]
Then $\xi_R$ is a Lipschitz function with $|\n \xi_R| \leq |\n \xi|(1-\chi_{D_R})$. 

By (\ref{gradient of the cut-off function}), the cut-off function $\phi$ satisfies
\[
|\n \phi| \leq \frac{C(n,k)}{\rho} \phi^{1-1/k}, 
\]
and $|\n \phi| = 0$ on $D_R$. Hence we have
\[
80^{-1} II = \int |\n u|^2 e^\xi |\n \phi|^2 \leq \frac{C(n,k)}{\rho^2} \int |\n u|^2 e^{\xi_R} \phi^{2-2/k}. 
\]
By integration by parts and the Cauchy-Schwarz inequality, we have
\[
\begin{split}
& \int |\n u|^2 e^{\xi_R} \phi^{2-2/k} \\
= & \int - u \Delta u e^{\xi_R} \phi^{2-2/k} - u \langle \n u, \n \xi_R \rangle e^{\xi_R} \phi^{2-2/k} - (2-2/k) u \langle \n u, \n \phi\rangle e^{\xi_R} \phi^{1-2/k} \\
\leq & \epsilon \int |\Delta u|^2 e^{\xi_R} \phi^2 + \epsilon^{-1} \int u^2 e^{\xi_R} \phi^{2-4/k} + \frac{1}{4} \int |\n u|^2 e^{\xi_R} \phi^{2-2/k} \\
&+ \int u^2 e^{\xi_R} |\n \xi_R|^2 \phi^{2-2/k} + \frac{1}{4} \int |\n u|^2 e^{\xi_R} \phi^{2-2/k} + C(k)\int u^2 e^{\xi_R} |\n \phi|^2 \phi^{-2/k},
\end{split}
\]
thus
\[
\begin{split}
& \int |\n u|^2 e^{\xi_R} \phi^{2-2/k} \\
\leq & 2 \epsilon \int |\Delta u|^2 e^{\xi_R} \phi^2 +  \int u^2 e^{\xi_R} (2\epsilon^{-1} \phi^{2-4/k} + 2 |\n \xi_R|^2 \phi^{2-2/k} + 2 C(k) |\n \phi |^2\phi^{-2/k}).
\end{split}
\]
Now choose $\epsilon>0$ small enough depending on $n, k$, we have the estimate
\begin{equation}\label{estimate of II}
II \leq \frac{1}{4} \int |\Delta u|^2 e^{\xi_R} \phi^2 + \frac{C(n, k)}{\rho^2 } \int u^2 e^{\xi_R} ( \phi^{2-4/k} +  |\n \xi_R|^2 \phi^{2-2/k} +  |\n \phi |^2\phi^{-2/k}).
\end{equation}

By (\ref{estimate of the time derivative by I and II}), (\ref{estimate of I}) and (\ref{estimate of II}), we have

\begin{equation}
\begin{split}
 &\partial_t \int u^2 e^\xi \phi^2 - \int u^2 e^\xi \phi^2 \partial_t \xi \\
\leq &  40\int  u^2 |\n \phi|^4 \phi^{-2} e^\xi + u^2 e^\xi |\Delta \phi|^2 + u^2 e^\xi  |\Delta \xi|^2 \phi^2 \\
& + C \int u^2 e^\xi \left( G^2 \phi^2 + |\n G |^2 G^{-1}  \phi^2  +   G  |\n \phi|^2 \right) \\
& + \frac{C(n, k)}{\rho^2 } \int u^2 e^{\xi_R} ( \phi^{2-4/k} +  |\n \xi_R|^2 \phi^{2-2/k} +  |\n \phi |^2\phi^{-2/k}).
\end{split}
\end{equation}
By the Cauchy-Schwarz inequality $|\n \xi_R|^2 \phi^{2-2/k} \leq G^2 \phi^2 + \chi_{D_{R+\rho}\backslash D_R}$.
Apply (\ref{gradient of the cut-off function}) and (\ref{Laplacian of the cut-off function}), we derive the following
\begin{equation}\label{estimate of time derivative}
\begin{split}
\partial_t \int u^2 e^\xi \phi^2 \leq  & \int u^2 e^\xi \phi^2 (\partial_t \xi + C G^2 + C |\n G|^2 G^{-1} + C |\Delta \xi|^2) \\
& + \frac{C(n,k)(1+K(R+\rho))}{\rho^2} \int_{D_{R+\rho} \backslash D_{R}} u^2 e^{\xi} \phi^{2-4/k},
\end{split}
\end{equation}
for some universal constant $C$.

\end{proof}

\subsection{$L^2$ exponential decay estimate}
As a direct application of Lemma \ref{time derivative of weighted l2 integral}, we show that $L^2$ solutions of the biharmonic heat equation with compactly supported initial data have exponential decay in $L^2$ norm. 
\subsubsection{The choice of weight functions}
Recall that 
\[
r(x) \leq f(x) \leq \Lambda(n) r(x), \quad |\n f|(x) \leq \Lambda(n), \quad when \quad r(x) \geq R_0,
\]
where $R_0$ depends on the geometry of $B(p,3)$. Let
\[
R_1 = \Lambda R_0, \quad then \quad B_{R_0} \subset D_{R_1}. 
\]
For any $R > R_1$, and WLOG assume $R > 1$. Define
\[
\xi (x,t) =
\begin{cases}
  - \frac{\left( 2R - R_1 \right)^{4/3}}{A(T - t) ^{1/3}}, & x \in D_{R_1} ,\\
  - \frac{\left( 2R - f(x) \right)^{4/3} \psi(f)}{A(T - t) ^{1/3}}, & x \in D_{\frac{3}{2}R}\backslash D_{R_1} ,\\
  - \frac{ \left( \frac{1}{2}R\right)^{4/3}}{A(T - t) ^{1/3}}, &  x \in M \backslash D_{\frac{3}{2} R}.\\
\end{cases}
\]

\[
G(x, t) = \Lambda^2(n)
\begin{cases}
   \frac{\frac{16}{9} \left( 2R - R_1 \right)^{2/3 }}{A^2(T - t) ^{2/3}}, & x \in D_{R_1} ,\\
   \frac{\left( 2R - f(x) \right)^{2/3 } [\frac{4}{3} \psi(f) - (2R - f) \psi'(f)]^2 }{A^2(T - t) ^{2/3}}, & x \in D_{\frac{3}{2}R} \backslash D_{R_1},\\
   \frac{\frac{16}{9} \left( \frac{1}{2}R\right)^{2/3 }}{A^2(T - t) ^{2/3}}, &  x \in M \backslash D_{\frac{3}{2}R},\\
\end{cases}
\]
where the function $\psi$ is defined by 
\[
\psi(f) = 1 + \frac{\frac{4}{3}(f -R_1)(\frac{3}{2}R - f)^2}{ (2R - R_1)(\frac{3}{2}R - R_1)^2} - \frac{\frac{4}{3}(f - R_1)^2(\frac{3}{2} R - f)}{ (\frac{1}{2}R) (\frac{3}{2}R - R_1)^2},
\]
for $ R_1 < f < \frac{3}{2}R$. Observe that $\xi$ is $C^1$ and $G$ is Lipschitz, and $|\n \xi(x, t)|^2 \leq G(x, t)$ for all $(x, t)$. By direct calculation we can verify that $\frac{2}{3} \leq \psi \leq \frac{4}{3}$ and $|\nabla \psi(f)| \leq C(n) R^{-1}$, similarly we can also verify that the term in the bracket in $G$ is also bounded with gradient  bounded by $C(n)R^{-1}$ for some constant $C(n)$, hence

\[
|\n G(x, t)| \leq
\begin{cases}
   0, & x \in D_{R_1} ,\\
   \frac{C(n)\left( 2R - f(x) \right)^{-1/3 }}{A^2(T - t) ^{2/3}}, & x \in D_{\frac{3}{2}R} \backslash D_{R_1},\\
  0, &  x\in M \backslash D_{\frac{3}{2}R}.\\ 
\end{cases}
\]

\[
|\n G|^2 G^{-1}(x, t) \leq 
\begin{cases}
   0, & x \in D_{R_1} ,\\
   \frac{  C(n) }{A^2(T - t) ^{2/3} (\frac{1}{2}R)^{4/3}}, & x \in D_{\frac{3}{2}} \backslash D_{R_1},\\
  0, &  x\in M \backslash D_{\frac{3}{2}R}.\\ 
\end{cases}
\]
Calculate that
\[
\begin{split}
\Delta \xi = & \frac{- \frac{4}{9} \left( 2R - f \right)^{-2/3} \psi |\n f|^2 + \frac{4}{3}\left( 2R - f \right)^{1/3} \psi \Delta f(x)  }{A(T - t) ^{1/3}} \\
& - \frac{+ \frac{8}{3}(2R - f)^{1/3}\psi' |\n f|^2 + (2R - f)^{4/3}\psi'' |\n f|^2 + (2R - f)^{4/3} \psi' \Delta f}{A(T - t) ^{1/3}}
\end{split}
\]
when $R_1 < f < \frac{3}{2}R$, and $\Delta \xi = 0$ elsewhere. 
By the properties of $f$ we can obtain the estimate
\[
|\Delta \xi (x,t)| \leq
\begin{cases}
  0, & x \in D_{R_1} ,\\
  \frac{C(n)R^{-2/3} + C(n)\frac{4}{3}\sqrt{1+ K(\frac{3}{2}R)}\left( 2R - f(x) \right)^{1/3 }  }{A(T - t) ^{1/3}}, & x \in D_{\frac{3}{2}R}\backslash D_{R_1} ,\\
  0, &  x\in M \backslash D_{\frac{3}{2}R}.\\ 
\end{cases}
\]

Next we show that we can choose $A$ large enough depending on $n$, so that the quantity
\begin{equation}\label{eqn: definition of N}
\mathcal{N}(\xi, G) : =\partial_t \xi + C G^2 + C |\n G|^2 G^{-1} + C |\Delta \xi|^2
\end{equation}
is non-positive. 

When $f(x) < R_1$, we can choose $A$ large enough depending on $n$, such that
\[
\mathcal{N} = -\frac{\frac{1}{3} (2R - R_1)^{4/3}}{A (T-t)^{4/3} } + \frac{C \Lambda^4 \rbrackets{\frac{4}{3}}^4 (2R - R_1)^{4/3} }{A^4(T-t)^{4/3}} \leq 0.
\]

When $R_1 < f(x) \leq \frac{3}{2}R$, for some constant $C(n)$ we have
\[
\begin{split}
\mathcal{N} \leq & -\frac{\frac{1}{3} (\frac{1}{2}R)^{4/3}}{A (T-t)^{4/3}  } + \frac{C(n) (2R )^{4/3 }}{A^4 (T-t)^{4/3}}  + \frac{C(n) }{A^2 (T-t)^{2/3} R^{4/3}} \\
&+ \frac{ C(n) (1+K(\frac{3}{2}R))(2R)^{2/3 }}{A^2 (T-t)^{2/3}}.
\end{split}
\]
Note that we have assumed $R>1$. When 
\begin{equation}\label{requirement for T}
 T-t < \frac{R}{\left( 1+K(\frac{3}{2}R) \right)^{3/2}},
 \end{equation}
we can choose $A$ large enough depending on $n$ such that $\mathcal{N} < 0$.

When $f > \frac{3}{2}R$, it is easy to check that $\mathcal{N} < 0$ when $A > 10C \Lambda(n)^4$.

\subsubsection{$L^2$ exponential decay.} Now suppose $u(x,t)$ is an $L^2$ solution of the biharmonic heat equation, and the initial data $u(x,0)$ is compactly supported in $D_R$. Suppose also that the Ricci lower bound $K(r)$ has at most quadratic growth in $r$. We derive exponential decay estimate for $u$ in $L^2$ norm, although this is not used in the following sections. 

Suppose (\ref{requirement for T}) holds for $T$ and $R$.
By lemma \ref{time derivative of weighted l2 integral} (with $\rho = R$ when choosing $\phi$), and by the choice of the weight function $\xi$, we have
\[
\partial_t \int u^2 e^\xi \phi^2 \leq   \frac{C(n,k)(1+K(2\rho))}{\rho^2} \int_{D_{2\rho} \backslash D_{\rho}} u^2 e^{\xi} \phi^{2-4/k}.
\]
Integrate on the time interval $[0,\frac{T}{2}]$,
\[
\int u^2 e^\xi \phi^2 (\frac{T}{2}) - \int u^2 e^\xi \phi^2 (0) \leq \frac{TC(n,k)(1+K(2\rho))}{\rho^2} \int_{D_{2\rho} \backslash D_{\rho}} u^2 e^{\xi} \phi^{2-4/k}.
\]
Suppose $K(r) \leq C (1+ r)^2 $ as $r \to \infty$, then we can 
let $\rho \to \infty$, the RHS above goes to $0$ since $u$ is $L^2$, hence
\[
\int u^2 e^\xi  (\frac{T}{2}) \leq \int u^2 e^\xi (0).
\]
Since $u(x,0)$ is compactly supported in $D_R$, we have  
\[
e^{-\frac{(\frac{1}{2}R)^{4/3}}{A(\frac{1}{2}T)^{1/3}}}\int_{M\backslash D_{2R}} u^2 (T/2) \leq e^{- \frac{R^{4/3}}{AT^{1/3}}} \int_{D_R} u^2(0),
\]
hence the exponential decay of $u$ in $L^2$ norm: 
\[
\int_{M\backslash D_{2R}} u^2 (T/2) \leq e^{- \frac{R^{4/3}}{2AT^{1/3}}} \int_{D_R} u^2(0).
\]


\section{Existence and decay estimates for the fundamental solution}
\subsection{Existence of the $L^2$ biharmonic heat kernel}
The existence of a fundamental solution of the biharmonic heat equation follows from semigroup theory, and the proof is the same as for the heat equation.  Let\rq{}s sketch it here and refer to \cite{G2009} for details. 

The Laplace operator $\Delta$ on $C_0^\infty(M)$ has a unique self-adjoint extension $\Delta_{W_0^{2}(M)}$ with domain
\[
W_0^2(M)  = \set{ u \in H_0^1(M) : \Delta u \in L^2 (M)}.
\]
Where $H_0^1(M)$ is the closure of $C_0^\infty(M)$ in $H^1(M)$. Define
\[
W_0^4(M) = \set{ u \in W_0^2(M) : \Delta u \in W_0^2(M)}.
\]
Define $L$  to be $\Delta_{W_0^2(M)}^2$ restricted on $W_0^4(M)$, then $L$ is the unique self-adjoint extension of $\Delta^2$.
\begin{proof}
For self-adjointness, we only need to show that the domain of $L^*$ is contained in the domain of $L$.
\[Dom(L^*) = \set {u \in L^2(M): \exists f \in L^2(M), s.t. \forall v \in Dom(L), \langle Lv, u\rangle_{L^2} = \langle v,f\rangle_{L^2}  }.
\]
If $u$ is in the domain of $L^*$, then it is a distributional solution of 
\[
\Delta(\Delta u)= f,
\]
for some $f\in L^2(M)$. Hence $\Delta u \in W_0^2(M)$ and $u \in W_0^4(M)$.

To show uniqueness, suppose $L_1$ is another self-adjoint extension of $\Delta^2$. Then the adjoint operator $(\Delta^2)^*$ is an extension of $L_1^* = L_1$. For any $u \in Dom(L_1)$, by the self-adjointness of $L^1$,
\[
L_1 u = L_1^* u = (\Delta^2 )^* u,
\]
hence for any test function $v$,
\[\langle L_1 u , v\rangle _{L^2} = \langle u, \Delta^2 v\rangle _{L^2} = \langle u, Lv\rangle_{L^2}, \]
then $L_1 = L$ since both are self-adjoint. 
\end{proof}

By standard semigroup theory, there is a semigroup $\set{e^{-tL}}_{t\geq0}$, such that for any $f\in L^2(M)$, $e^{-tL} f$ is an $L^2(M)$ solution of the Cauchy problem
\[
(\partial_t -L)e^{-tL}f = 0, \quad \lim_{t \to 0^+} e^{-tL} f = f.
\]
Regularity of the solution $e^{-tL}f$ follows from our $L^2$ estimates in Section \ref{section: l2 estimates and mean value inequality} and Sobolev embedding theorems. Thus by Riesz representation theorem (and by symmetrization), there exists a smooth kernel function $b(x, y,t) : M \times M \times (0, \infty)$, such that 
\[
e^{-tL}f(x) = \int_M b(x, y, t)f(y) dy.
\]
We call the function $b(x, y, t)$ \textbf{biharmonic heat kernel} on $M$.

\begin{rem} The regularity of the semigroup can be proved by an alternative way. One can obtain $L^2$ estimates of $\Delta^{2k} (e^{-t\Delta^2 }f )$ by method of \cite{G2009}, which uses the spectral resolution. Then one can estimate $\Delta^k (e^{-t\Delta^2 }f )$ by similar integration by parts argument as above, which again requires Ricci lower bound to control the Laplacian of cut-off functions. 
\end{rem}

In the following, we show that the biharmonic heat kernel $b(x,y,t)$ has exponential decay under some geometric assumptions, thus the domain of the biharmonic heat semigroup can be enlarged. 

\subsection{Exponential decay for the fundamental solution.}
Now we can use the method of \cite{G2009} and the ingredients we developed in the previous sections to obtain exponential decay estimates for the biharmonic heat kernel.
\subsubsection{$L^2$ estimates for $b(p,y,t)$.}
 First choose a suitable weight function $\xi$, and $G$ depending on $\xi$. Let $f$ be a distance-like function from Lemma \ref{global distance-like function} or \ref{global distance-like function on Ricci nonnegative manifolds},  $0\leq t < T$, define
\begin{equation}\label{xi with exponential decay}
\xi = \begin{cases}
0, & D_R; \\
-\frac{S^{-8/3}(f-R)^4[1 - \frac{8}{3}S^{-1}(f-R-S)]}{A(T-t)^{1/3}} , & D_{R+S} \backslash D_R;\\
- \frac{(f - R)^{4/3}}{A(T-t)^{1/3}}, & M \backslash D_{R+S}.
\end{cases}
\end{equation}
\[
G = \begin{cases}
0, & D_R; \\
\frac{\Lambda(n)^2 S^{-16/3}(f-R)^6[4 - \frac{32}{3}S^{-1}(f-R-S) - \frac{8}{3}S^{-1}(f-R)]^2}{A^2(T-t)^{2/3}} , & D_{R+S} \backslash D_R;\\
 \frac{\frac{16}{9}\Lambda(n)^2(f - R)^{2/3}}{A^2(T-t)^{2/3}}, & M \backslash D_{R+S}.
\end{cases}
\]
We can check by direct calculation that $\xi$ is $C^1$, $G$ is Lipschitz, and $|\n \xi|^2 \leq G$. 

On $(D_{R+S} \backslash D_{R}) \cap (M \backslash B(p, R_0) )$, where $R_0$ is the constant from Lemma \ref{global distance-like function}, observe that the terms in the brackets in the definition of $\xi$ and $G$ are bounded by uniform positive constants from both above and below, and their gradients are bounded by $C(n)S^{-1}$, then by direct calculation we can derive that the quantity $\mathcal{N}$ defined in (\ref{eqn: definition of N}) satisfies
\begin{equation}\label{calculation of N when R is large}
\begin{split}
\mathcal{N} = & \partial_t \xi + C G^2 + C |\n G|^2 G^{-1} + C |\Delta \xi|^2 \\
 \leq & - \frac{\frac{1}{3} S^{-8/3} (f- R)^4}{A(T-t)^{4/3}} + \frac{C(n) S^{-32/3} (f-R)^{12}}{A^4 (T-t)^{4/3}} + \frac{C(n) S^{-16/3} (f-R)^4}{A^2 (T-t)^{2/3}} \\
& + \frac{C(n) (1+K(R+S)) S^{-16/3} (f-R)^6}{A^2 (T-t)^{2/3}},
\end{split}
\end{equation}
for some constant $C(n)$, where we have applied Lemma \ref{global distance-like function} (or Lemma \ref{global distance-like function on Ricci nonnegative manifolds} if $Ric \geq 0$) to control $\Delta f$ and hence $|\Delta \xi|$, the calculation is similar to that of the previous section. 

On $(D_{R+S} \backslash D_{R}) \cap B(p, R_0) $, the last term in (\ref{calculation of N when R is large}) becomes
\[
\frac{C(n) R_0^2 (1+K(C(n)R_0)) S^{-16/3} (f-R)^6}{A^2 (T-t)^{2/3} f^2 } \leq \frac{C(n) R_0^2 (1+K(C(n)R_0)) S^{2/3} }{A^2 (T-t)^{2/3} R^2 } .
\]

We have $\mathcal{N}< 0$ when we choose $A$ large enough depending on $n$, and require that
\begin{equation}\label{upper bound for T 1}
T \leq \min \set{S^4, \quad \frac{S}{(1+K(R+S))^\frac{3}{2}}, \quad \frac{R^3S}{R_0^3(1+K(C(n)R_0))^\frac{3}{2}}}.
\end{equation} 
On $(M \backslash D_{R+S}) \cap (M \backslash B(p, R_0))$,
\[
\begin{split}
\mathcal{N} \leq & - \frac{\frac{1}{3} (f- R)^{4/3}}{A(T-t)^{4/3}} + \frac{C(n) (f-R)^{4/3}}{A^4 (T-t)^{4/3}} + \frac{C(n) (f-R)^{-4/3}}{A^2 (T-t)^{2/3}} \\
& + \frac{C(n) (1+K(f)) (f-R)^{2/3}}{A^2 (T-t)^{2/3}},
\end{split}
\]
for some constant $C(n)$; and on $(M \backslash D_{R+S}) \cap B(p, R_0)$ we have
\[
\begin{split}
\mathcal{N} \leq & - \frac{\frac{1}{3} (f- R)^{4/3}}{A(T-t)^{4/3}} + \frac{C(n) (f-R)^{4/3}}{A^4 (T-t)^{4/3}} + \frac{C(n) (f-R)^{-4/3}}{A^2 (T-t)^{2/3}} \\
& + \frac{C(n) R_0^2(1+K(C(n)R_0)) (f-R)^{2/3}}{A^2 (T-t)^{2/3}f^2}.
\end{split}
\] 
In this case, we can make $\mathcal{N} < 0$ by choosing $A$ large enough depending on $n$, and require that
\begin{equation}\label{upper bound for T 2}
T \leq \min \set{S^4, \quad \inf_{r \geq R+S}  \frac{(r - R)}{(1+K(r))^\frac{3}{2}}, \quad \frac{(R+S)^3S}{R_0^3(1+K(C(n)R_0))^\frac{3}{2}} }.
\end{equation}
Clearly this requirement for $T$ is meaningful only when 
\begin{equation}\label{requirement for K(r)}
K(r) \leq C (1+r)^\frac{2}{3},
\end{equation}
for some constant $C$. 

   Then for any solution $u$ of the biharmonic heat equation that is $L^2$, apply Lemma \ref{time derivative of weighted l2 integral} to get
\[ \partial_t \int u^2 e^\xi \varphi_\rho^2 \leq \frac{C(n,k)(1+K(2\rho))}{\rho^2}\int_{D_{2\rho} \backslash D_\rho} u^2 e^\xi \varphi_\rho^{2-4/k},\]
where $\rho > \Lambda R_0$, $\varphi_\rho$ is a cut-off function constructed as in Section \ref{construction of cut-off function}, which is supported on $D_{2\rho}$ and equals $1$ on $D_\rho$. Integrate in time and then let $\rho \to \infty$, by the assumption (\ref{requirement for K(r)}) on $K(r)$, and the assumption that $u$ is in $L^2$,  
we get
\[
\int_M u^2 e^\xi (t) \leq \int_M u^2e^\xi (0) .
\]
By the mean value inequality Lemma \ref{lem: mean value inequality}, we have
\begin{equation}\label{eqn: control u(p,t)}
\begin{split}
u(p, t)^2 \leq  \gamma_p(R,t) \int_0^t \int_{D_R} u^2 
\leq  \gamma_p(R,t) \int_0^t \int_M u^2 e^\xi \leq t \gamma_p(R,t) \int_M u(x,0)^2 e^{\xi(x,0)},
\end{split}
\end{equation}
where
\begin{equation}\label{definition of gamma}
\gamma_p(R, t) = \frac{C(n, \mu)  (\max \set{R_0, R}^2\max \set{1, K_p(R)} + \frac{R^2}{\sqrt{t}})^{c(\mu)}e^{C(n,\mu)\sqrt{K_p(R)}R} }{ R^4  Vol(D_R)},
\end{equation}
where we specified the base point $p$ of the function $K(r)$.

Now, we take a particular solution of the biharmonic heat equation:  
\begin{equation}\label{eqn: definition of u(x,s)}
u(x, s) = \int_M b(x, y,s)b(p, y, t) e^{-\xi(y,0)} \phi(y) dy,
\end{equation}
where $0\leq \phi \leq 1$ is an arbitrary cut-off function. 
By properties of the fundamental solution, 
\[
u(x,0) = b(x,p,t)e^{-\xi(x, 0)} \phi(x),
\]
which is compactly supported, hence $u(x,t)$ is $L^2$. 
Then by (\ref{eqn: control u(p,t)}) and (\ref{eqn: definition of u(x,s)})
\[
\left(\int_M b(p, y,t)^2 e^{-\xi(y,0)} \phi(y) dy \right)^2 = u(p,t)^2 \leq t \gamma_p(R,t) \int_M b(x,p,t)^2 e^{-\xi(x,0)} \phi^2(x)dx.
\]
By the symmetry of $b(x,y,t)$ and the fact that $\phi \leq 1$,
\[
\int_M b(p, y,t)^2 e^{-\xi(y,0)} \phi(y) dy \leq t \gamma_p(R,t).\]
Then we can take $\phi \to 1$ since the RHS is bounded. Note that  
 $\xi(x,0) = 0$ in $D_R$, thus
\begin{equation}\label{eqn: pre-weighted l2 estimate for b}
\int_{D_R} b(p,y,t)^2 dy \leq \int_M b(p, y,t)^2 e^{-\xi(y,0)} dy \leq t \gamma_p(R,t),
\end{equation}
Note that we can take $t=T$ in (\ref{eqn: pre-weighted l2 estimate for b}) since the RHS is bounded, and $T$ can be any positive number satisfying (\ref{upper bound for T 2}). Therefore we have in fact proved that
\begin{equation}\label{eqn: weighted l2 estimate for b}
E_p(t) : = \int_M b(p, y,t)^2 e^{\eta_p(y,t)}  dy < t\gamma_p(R,t),
\end{equation}
for 
\[ 0< t \leq \min \set{S^4, \quad \frac{R^3S}{R_0^3(1+K(C(n)R_0))^\frac{3}{2}}, \quad \inf_{r \geq R+S} \frac{(r - R)}{(1+K(r))^\frac{3}{2}}},
\]
where $\eta_p$ is defined by
\[
\eta_p= \begin{cases}
0, & D_R; \\
\frac{S^{-8/3}(f_p-R)^4[1- \frac{8}{3}S^{-1}(f-R-S)]}{At^{1/3}} , & D_{R+S} \backslash D_R;\\
 \frac{(f_p - R)^{4/3}}{At^{1/3}}, & M \backslash D_{R+S}.
\end{cases}
\]
Note that $\eta_p$ is related to $\xi$ in the following way: if we denote $\xi$ in (\ref{xi with exponential decay}) as $\xi(f(x), t, T, R, S)$, then 
\[ \eta_p = - \xi(f_p(x), 0, t, R, S).
\]
Choose 
\[
G_\eta = \begin{cases}
0, & D_R; \\
\frac{\Lambda(n)^2 S^{-16/3}(f_p-R)^6[4 - \frac{32}{3}S^{-1}(f-R-S) - \frac{8}{3}S^{-1}(f-R)]^2}{A^2t^{2/3}} , & D_{R+S} \backslash D_R;\\
 \frac{\frac{16}{9}\Lambda(n)^2(f_p - R)^{2/3}}{A^2t^{2/3}}, & M \backslash D_{R+S}.
\end{cases}
\]
Then if we calculate the quantity $\mathcal{N}$ with $\eta_p$ and $G_\eta$, we will have $\mathcal{N}< 0$ under the same restraints for the parameters as for $\xi$ and $G$. 

\subsubsection{Pointwise exponential decay.}
By the semigroup property
\[
b(p,q,t) = \int_M b(p,x,\frac{t}{2}) b(x, q, \frac{t}{2}) dx.
\]
By the triangle inequality
\[
d(p,q) \leq d(p,y ) + d(q, y) \leq f_p(y) + f_q(y),
\]
hence
\[
(d(p,q) - 2R)_+ \leq (f_p(y ) - R )_+  + (f_q(y ) - R )_+, \quad \forall y\in M,
\]
we can choose $c(n)$ small enough such that
\[
\frac{c(n) (d(p,q) - 2R -2S)_+^{4/3}}{t^{1/3}} \leq \eta_p(y,t) + \eta_q(y,t), \quad \forall y \in M.
\]
Then by Holder inequality
\[
\begin{split}
| b(p, q,t) | \leq & \int | b(p,x,t/2) | e^{\eta_p(x,t/2)} | b(q, x, t/2) | e^{\eta_q(x, t/2)} e^{-\frac{c(n) (d(p,q) - 2R -2S)_+^{4/3}}{t^{1/3}} } dx \\
\leq & \sqrt{ E_p(t/2) E_q(t/2) } e^{-\frac{c(n) (d(p,q) - 2R -2S)_+^{4/3}}{t^{1/3}} },
\end{split}
\]
where $E_p(t)$ and $E_q(t)$ have been estimated in (\ref{eqn: weighted l2 estimate for b}).

 In conclusion, we have proved: 
\begin{thm}\label{exponential decay estimate general case}
   Suppose on a complete noncompact Riemannian manifold we have $Ric (y) \geq - K_x(d(y,x))$, where $K_x(r)$ are nondecreasing functions for each point $x = p,q$. Assume that $K_x(r) \leq C (1+r)^{2/3}$ for $x = p,q$. Then for any $R >0 $, $S >0$, the biharmonic heat kernel satisfies
\[
| b(p,q,t) | \leq t \sqrt{\gamma_p(R,t) \gamma_q(R,t)} e^{-\frac{c(n) (d(p,q) - 2R -2S)_+^{4/3}}{t^{1/3}} },
\] 
when
\[
t \leq \bar{T} : =\min_{x = p,q}  \set{S^4,  \quad \frac{R^3S}{R_0(x)^3(1+K(C(n)R_0(x)))^\frac{3}{2}},\quad  \inf_{r \geq R+S} \frac{(r - R)}{(1+K_x (r))^\frac{3}{2}}},
\]
where $\gamma_x(R,t)$ is defined in (\ref{definition of gamma}), for $x= p, q$, $R_0(x)$ is a constant depending on $n$, $K_x(1)$ and $Vol(B(x, 1))$.
\end{thm}
If we have Ricci curvature uniformly bounded from below $Ric \geq -K$, and uniformly noncollapsing unit geodesic balls, then the coefficients in the above estimates becomes uniform (no longer depends on $p,q$).
\begin{thm}\label{exponential decay estimate with Ricci lower bound} 
Suppose on a complete noncompact Riemannian manifold $M$ we have $Ric (x) \geq - K$, where $K\geq 0$ is a constant, and $Vol (B(x, 1)) > v>0$ for all $x\in M$. Then for any $S >0$, the biharmonic heat kernel satisfies for any $p,q\in M$,
\[
| b(p,q,t) | \leq \frac{\Gamma } {\sqrt{Vol(B(p,t^\frac{1}{4})) Vol(B(q,t^\frac{1}{4}))}} e^{-\frac{c(n) (d(p,q)  -2S)_+^{4/3}}{t^{1/3}} },
\] 
when
\[
0< t \leq \min \set{ \frac{S^4}{R_0^{12} (1+K)^6}, \quad \frac{S}{(1+K)^\frac{3}{2}}},
\]
where 
\[
\Gamma =C(n, \mu) \left( \max \set{R_0, t^\frac{1}{4}}^2 \max \set{1, K} + 1 \right)^{c(\mu)} e^{C(n, \mu) \sqrt{K} t^\frac{1}{4}},
\]
$R_0$ is a constant depending on $n$, $K$ and $v$.
\end{thm}
\begin{proof}
In this case, $R_0$ is a constant depending only on $n$, $K$ and $v$, WLOG we can take $R_0 \geq 1$.  Then for each
\[
0 < t < \min \set { \frac{S^4}{R_0^{12} (1+K)^6}, \quad \frac{S}{(1+K)^\frac{3}{2}}},
\]
take $R^4 = t$, we can check that the requirement for $t$ in Theorem \ref{exponential decay estimate general case} holds. Moreover, the quantity $\gamma$ now is independent of base point and simplifies to the form
\[
\gamma_p(R, t) \leq \frac{C(n, \mu) \left( \max \set{R_0, t^\frac{1}{4}}^2 \max \set{1, K} + 1 \right)^{c(\mu)} e^{C(n, \mu) \sqrt{K} t^\frac{1}{4}} }{ t Vol (B(p, t^\frac{1}{4}))},
\]
where we have used also the volume comparison theorem to handle the volume term.
\end{proof}
If the manifold $M$ has $Ric \geq 0$, then we have a Li-Yau type estimate for the biharmonic heat kernel, which is Theorem \ref{thm: estimate of biharmonic heat kernel nonnegative Ricci curvature} stated in the Introduction. 
  \begin{proof}[Proof of Theorem \ref{thm: estimate of biharmonic heat kernel nonnegative Ricci curvature}]
    The proof is the same as in the proof of Theorem \ref{exponential decay estimate general case}, except that in this case we use the distance-like function from Lemma \ref{global distance-like function on Ricci nonnegative manifolds}. For each point $p$, we can take $R_0(p) = 0$ since the estimate of $\Delta f_p$ is in the same form on $M\backslash \set{p}$. Using the estimate 
    \[|\Delta f| \leq \frac{C(n)(1 - \ln \nu_p(f))}{f}\] 
in the calculation of the quantity $\mathcal{N}$, the requirement for $T$ in (\ref{upper bound for T 1}) and (\ref{upper bound for T 2}) becomes 
\[
  T \leq \frac{S^4}{(1- \ln \nu_p(R+S))^3}.
\] 
Hence for any $T > 0$ we can take $R = T^{1/4}$, and take $S$ such that
\[
  T = \psi_p(S) \leq  \frac{S^4}{ (1- \ln \nu_p(2S))^3} , 
\]
then the quantity $\Gamma$ in Lemma \ref{lem: mean value inequality} is controlled by
\[
  \Gamma  = C(n, \mu)(\sqrt{1 - \ln \nu_p(R)} + \frac{R^2}{\sqrt{T}})^{c(\mu)} \leq C(n, \mu) (1- \ln \nu_p(T^\frac{1}{4}))^{c(\mu)},
\] 
note that it may depend on the base point $p$, so we denote it by $\Gamma_p$ to specify the base point. 

Note also that
 $d(p, q) \leq (d(p, q ) - R - S)_+ + R + S$,
where $R = T^{1/4}$ and 
$
\frac{S^{4/3}}{T^{1/3}} = (\frac{\psi_p^{-1}(T)}{T})^\frac{1}{3}
$, we have the estimate
\[
| b(p,q,T) | \leq \frac{C(n)\Gamma_p^\frac{1}{2} \Gamma_q^\frac{1}{2} e^{c(n) ( \frac{\psi_p^{-1}(T)^4 }{ T})^{1/3} } e^{c(n) ( \frac{\psi_q^{-1}(T)^4 }{ T})^{1/3} }  } {\sqrt{Vol(B(p,T^{1/4})) Vol(B(q,T^{1/4}))}} e^{-\frac{c(n) (d(p,q) )^{4/3}}{T^{1/3}} },
\]
Finally replace $T$ by $t$ to finish the proof. 
  \end{proof}

\section{Uniqueness for the biharmonic heat equation and corollaries}
In this section we first prove a uniqueness theorem for the biharmonic heat equation. This is a fourth order generalization of the methods in \cite{KL},\cite{G1987} and \cite{HL2020}. Then as a corollary we prove the conservation law for the biharmonic heat kernel.
\subsection{Proof of uniqueness}
\begin{proof}[Proof of Theorem \ref{uniqueness theorem}]
We will use Lemma \ref{time derivative of weighted l2 integral} again, so we begin with the definition of the weight functions. Let $f$ be a distance-like function from Lemma \ref{global distance-like function} or \ref{global distance-like function on Ricci nonnegative manifolds}, define 
\begin{equation}\label{definition of xi in the proof of uniqueness}
\xi (x,t) = \xi(f(x),t)=
\begin{cases}
  - \frac{\left(f(x) - R \right)^{4/3}}{A(T - t) ^{1/3}}, & x \in M \backslash D_{2R} ,\\
  - \frac{R^{-8/3} \left( f(x) - R \right)^{4}[1 - \frac{8}{3}R^{-1}(f(x) - 2R)]}{A(T - t) ^{1/3}}, &  x \in D_{2R} \backslash D_{R},\\
  0, &  x\in D_{R}.
\end{cases}
\end{equation}
\[
G(x,t) : = C(n)
\begin{cases}
   \frac{ \frac{16}{9} \left( f(x) - R \right)^{2/3} }{A^2(T - t) ^{2/3}}, &x \in M \backslash D_{2R},\\
   \frac{  R^{-16/3} \left( f(x) - R \right)^6 [ 4 - \frac{32}{3}R^{-1}(f(x) - 2R) - \frac{8}{3}R^{-1}(f(x) - R) ]^2 }{A^2(T - t) ^{2/3}}, &  x \in D_{2R} \backslash D_{R},\\
  0, &  x\in D_{R},
\end{cases}
\]
for $C(n)$ large enough, we have $|\n \xi|^2 \leq G$. In this proof, WLOG we can assume that $R > R_0$. 

When $ f> 2R$, we have
\[
\begin{split}
\mathcal{N} = & \partial_t \xi + C G^2 + C |\n G|^2 G^{-1} + C |\Delta \xi|^2 \\
\leq &- \frac{\frac{1}{3} (f - R)^{4/3}}{A (T-t)^{4/3}} +  \frac{C(n) (f - R)^{4/3 }}{A^4 (T-t)^{4/3}} + \frac{C(n)(f-R)^{-4/3}}{A^2 (T-t)^{2/3}} \\
& + \frac{C(n) (f - R)^{\frac{2}{3}} (1+K(\Lambda f))}{A^2 (T-t)^{2/3}},
\end{split}
\]
where we have used the bound on $\Delta f$ in Lemma \ref{global distance-like function} to control $|\Delta \xi|$. 

Require that
\begin{equation}\label{requirement for T-t}
T - t \leq  \frac{R}{(1+K(4 \Lambda R))^\frac{3}{2}},
\end{equation}
then we can choose $A$ large enough depending on $n$ and the constants in the function $K(r)$, such that $\mathcal{N}<0$.

When $R \leq f \leq 2R$, note that the terms in brackets are bounded from above and below, and has gradient bounded by $C(n)R^{-1}$, hence similarly as above we can derive that
\[
\begin{split}
\mathcal{N} \leq & - \frac{c(n) R^{-8/3}  (f - R)^{4 }}{A (T-t)^{4/3}} +  \frac{C(n) R^{-32/3}  (f - R)^{12 }}{A^4 (T-t)^{4/3}} + \frac{C(n) R^{-16/3} (f -R)^4}{A^2 (T-t)^{2/3} }\\
&+ \frac{C(n) R^{-16/3}  (f - R)^6(1+K(2R))}{A^2 (T-t)^{2/3}}. \\
\end{split}
\]
Under the same assumption (\ref{requirement for T-t}), we can choose $A$ large enough depending on $n$, such that $\mathcal{N} < 0$.

Consequently by Lemma \ref{time derivative of weighted l2 integral} we have
\begin{equation}\label{final estimate of time derivative}
\frac{d}{dt} \int u^2 e^\xi \phi^2 \leq \frac{C(n,k)(1+K(4R))}{4R^2} \int_{D_{4R} \backslash D_{2R}} u^2 e^{\xi} \phi^{2-4/k},
\end{equation}
where $\phi$ is a cut-off function constructed in section \ref{construction of cut-off function}, with support on $D_{4R}$ and $\phi= 1$ on $D_{2R}$.

For $T > \tau$, we have 
\begin{equation}\label{Newton-Leibniz formula}
\frac{1}{T} \int u^2 e^\xi \phi^2 - \frac{1}{\tau} \int u^2 e^\xi \phi^2 = \int_{\tau}^{T} t^{-1} \frac{d}{dt} \int u^2 e^\xi \phi^2 - \int_{\tau}^{T} t^{-2} \int u^2 e^\xi \phi^2 .
\end{equation}
Apply Young\rq{}s inequality to the RHS of (\ref{final estimate of time derivative}), 
\[
\begin{split}
\frac{d}{dt} \int u^2 e^\xi \phi^2 \leq & t^{-1} \int_{D_{4R} \backslash D_{2R}} u^2 e^\xi \phi^2 \\
&+ \frac{2}{k}\left( \frac{k-2}{k} \right)^\frac{k}{2}\left(\frac{C(n,k)(1+K(4R))}{R^2}\right)^\frac{k}{2} t^{\frac{k}{2}-1} \int_{D_{4R} \backslash D_{2R}} u^2 e^\xi ,
\end{split}
\]
where the first term on the RHS is canceled when we plug it into (\ref{Newton-Leibniz formula}).
Hence
\[
\begin{split}
\frac{1}{T} \int u^2 e^\xi \phi^2 - \frac{1}{\tau} \int u^2 e^\xi \phi^2 \leq & C(n,k)\left(  \frac{1+K(4R)}{R^2}\right)^{\frac{k}{2}} \int_{\tau}^{T} t^{\frac{k}{2} - 2} \int_{D_{4R} \backslash D_{2R}} u^2e^{\xi} .\\
\end{split}
\]
Choose $k> 4$ large enough such that $\frac{k}{2}-2 > a$, then by the assumption on $u$, we have
\[
\begin{split}
\frac{1}{T} \int u^2 e^\xi \phi^2 - \frac{1}{\tau} \int u^2 e^\xi \phi^2 \leq  C(n,k)\left(  \frac{1+K(4R)}{R^2}\right)^{\frac{k}{2}} T^{\frac{k}{2} -2 -a} e^{L(4R) + \xi(2R,\tau)}.
\end{split}
\]
By requiring
\begin{equation}\label{requirement for T - tau}
T -\tau \leq \frac{R^4}{8 A^3 L(4R)^3},\quad and \quad T- \tau \leq \frac{ R^4}{k^3 A^3 \ln \left( 1+ K(4R)\right)^3},
\end{equation}
 we can make 
\[
\left(  \frac{1+K(4R)}{R^2}\right)^{\frac{k}{2}}  e^{L(4R) + \xi(2R,\tau)} \leq \frac{1}{R^k}. 
\] 
Therefore
\begin{equation}\label{induction formula}
 \frac{1}{T} \int_{D_{R}} u^2   - \frac{1}{\tau} \int_{D_{4R}} u^2 
\leq  C(n,k) \frac{1}{R^k} T^{\frac{k}{2} -2- a} .
\end{equation}

For $T < c_0$ and $R> (2c_0)^{-1}$, take a sequence of $R_i = 4^i R$, and a sequence $\tau_i$ with $\tau_0 = T$,  take $ \tau_{i+1} $ such that
\[
 \tau_i -\tau_{i+1} \leq \min \set{ \frac{R_i}{(1+ K(\Lambda R_{i+1}))^{3/2}}, \quad \frac{R_i^4}{8 A^3 L(R_{i+1})^3} },
\]
also note that the second inequality in (\ref{requirement for T - tau}) holds when $R$ is large. 
By (\ref{induction formula}), we have
\[
 \frac{1}{\tau_i} \int_{D_{R_i}} u^2   - \frac{1}{\tau_{i+1}} \int_{D_{R_{i+1}}} u^2 
\leq  C(n,k) \frac{1}{R_i^k}\tau_i^{\frac{k}{2} -2- a} .
\]
Summing over the index $i$, we have

\[
 \frac{1}{T} \int_{D_{R}} u^2   - \frac{1}{\tau_{N}} \int_{D_{R_{N}}} u^2 
\leq  C(n,k) \sum_{i=0}^{N-1} \frac{1}{R_i^k}\tau_i^{\frac{k}{2} -2- a} .
\]
By the assumptions on $L(r)$ we have
\[
\sum_{i = 0}^ \infty \frac{R_i}{(1+ K(\Lambda R_{i+1}))^{3/2}  } = \infty,
\]
\[
\sum_{i = 0}^ \infty \frac{R_i^4}{L(R_{i+1})^3} = \infty,
\]
 thus we can let $\tau_N \to 0$ for some finite number $N$, which may depend on $R$. 
Then by Lemma \ref{strong vanishing as t goes to 0} below, we get
\[
 \frac{1}{T} \int_{D_{R}} u^2  
\leq  C(n,k) T^{\frac{k}{2} -2- a} \sum_{i=0}^{N-1} \left(  \frac{1+K( R_{i+1})}{(R_{i+1} - R_i)^2}\right)^{\frac{k}{2}} \leq C(n,k)\frac{T^{\frac{k}{2} -2- a}}{R^k} .
\]
Take $R \to \infty$ finishes the proof since $D_R$ exhausts the entire manifold. 
\end{proof}

\begin{lem}\label{strong vanishing as t goes to 0}
Suppose $\|u^2(x,t)\|_{L^2(\Omega)} \to 0$ as $t \to 0$ for any compact set $\Omega \subset M $. Then for any $R>0$ we have 
\[
\lim_{t \to 0} \frac{1}{t} \int_{D_R} u^2 = 0. 
\]
\end{lem}
\begin{proof}
Let $m > 4$. For any $R > 0$, since $D_{2R} $ is compact, we can construct a cut-off function $\phi$ supported on $D_{2R}$ with $\phi = 1$ on $D_R$, and $|\n \phi| \leq C\phi^{1-1/m}$, $|\Delta \phi| \leq C \phi^{1-2/m}$, where the constants may depend on $m$ and the geometry of $D_{2R}$. By integration by parts and the Cauchy-Schwarz inequality, we have
\[
\begin{split}
\frac{d}{dt}\int u^2 \phi^2 = & \int - 2 u \Delta^2 u \phi^2 \\
= & \int - 2 |\Delta u|^2 \phi^2 - 8\phi  \langle \n u, \n \phi \rangle \Delta u - 4u \phi \Delta u \Delta \phi - 4 u \Delta u |\n \phi|^2    \\
\leq & - \int |\Delta u|^2 \phi^2 + C \int |\n u|^2 |\n \phi^2 | + u^2 (|\Delta \phi |^2 + |\n \phi|^4 \phi^{-2}), \\
\end{split}
\]
where
\[
\begin{split}
\int |\n u|^2 |\n \phi^2 | \leq & C \int |\n u|^2 \phi^{2-2/m} \\
= & - C \int u \Delta u \phi^{2-2/m} + (2-2/m) \phi^{1-2/m}u \langle \n u, \n \phi \rangle \\
\leq & \frac{1}{2}\int |\Delta u|^2 \phi^2 + \frac{1}{2} \int |\n u|^2 |\n \phi|^2 + C \int u^2 \phi^{1-4/m} +u^2 \phi^{2-4/m} .
\end{split}
\]
Hence 
\[
\frac{d}{dt} \int u^2 \phi^2 \leq C \int u^2 \phi^{1-4/m}. 
\]
Integrate on $[0, t]$, and use the Holder inequality to get
\[
\int u^2 \phi^2 (t) \leq C \int_0^t \int u^2 \phi^{1-4/m} \leq C t \sup_{s \in(0,t)} \int_{D_{2R}} u^2(s).
\]
The conclusion then follows from the assumption $\|u(x,t)\|_{L^2(D_{2R})} \to 0 $ as $t \to 0$.
\end{proof}

\subsection{Proof of the conservation law}
The conservation law for the biharmonic heat kernel, i.e. $\int_M b(x, y, t) dy \equiv 1$, follows directly from the uniqueness theorem and the exponential decay estimate Theorem \ref{exponential decay estimate general case}.
\begin{proof}[Proof of Corollary \ref{conservation law}]
In this proof we use $C$ to denote a constant that may change from line to line. 
By the curvature assumption $Ric \geq -K$ and the volume comparison theorem, the geodesic balls on $M$ has at most exponential volume growth. 
  By Theorem \ref{exponential decay estimate with Ricci lower bound}, $b(x, y, t)$ is integrable on $M$ for any $t> 0$, and $u(x, t):=\int_M b(x, y, t) dy$ is a solution of the biharmonic heat equation for $0< t< \bar{T}$, and $\lim_{t \to 0}u(x, t) = 1$, where
\[
\bar{T} =\min \set{ \frac{S^4}{R_0^{12} (1+K)^6}, \quad \frac{S}{(1+K)^\frac{3}{2}}}> 0.
\]  
By the volume comparison theorem, for any $x$ and for $ 0 < t < 1$,
\[
Vol (B(x, t^\frac{1}{4})) \geq \frac{Vol(B(x, 1))}{V_{-K}(1)} V_{-K}(t^\frac{1}{4}) \geq c(n, K, v) t^{n/4},
\]
where $V_{-K}(r)$ denotes the volume of a ball with radius $r$ on the space-form with constant sectional curvature $-K$.
Thus $u(x, t)$ satisfies 
  \[
  |u(x, t)| \leq \frac{C}{t^{\frac{n}{4}}},
  \]
when $t$ is sufficiently small.
Hence $u(x, t)$ satisfies the assumption (\ref{uniqueness class}), then by Theorem \ref{uniqueness theorem}, it has to be the constant solution $1$. Since $S$ can be any positive number, the corollary is prove for any $t > 0$.
\end{proof}

\section{$L^\infty$ estimates for the biharmonic heat equation}
In this section we prove a uniform $L^\infty$ estimate for the biharmonic heat equation on Ricci nonnegative manifolds with maximal volume growth, which can be viewed as a maximum principle for entire solutions. We also prove this type of estimate on closed manifolds.
\subsection{$L^\infty$ estimate on Ricci nonnegative manifolds with maximal volume growth}
Now we prove Corollary \ref{cor: uniform l-infinity estimate} in the Introduction.
\begin{proof}[Proof of Corollary \ref{cor: uniform l-infinity estimate}]
  By Theorem \ref{uniqueness theorem}, $u(x, t) = \int_M b(x, y, t) u(y, 0) dy$. By Theorem \ref{thm: estimate of biharmonic heat kernel nonnegative Ricci curvature} and the volume comparison theorem
  \[
    |b(x, y, t) | \leq \frac{C}{Vol B(x, t^{1/4})} \left( 1 + \frac{d(x, y)}{t^{1/4}}\right)^\frac{n}{2} e^{-\frac{c d(x, y)^{4/3}}{t^{1/3}}},
  \]
  where $C$ is a constant depending on $n$ and $v$. 
  By direct calculation 
  \[
  \int_M |b(x, y, t)| dy \leq   \int_0^\infty \frac{C A(\partial B(x,r))}{Vol B(x, t^{1/4})} \left( 1 + \frac{r}{t^{1/4}}\right)^\frac{n}{2} e^{-\frac{c r^{4/3}}{t^{1/3}}} dr. 
  \]
  Let $r = t^{1/4}s$, 
  \[
    \begin{split}
    \int_M |b(x, y, t)| dy \leq & C \int_0^\infty \frac{\frac{d}{ds} Vol B(x, t^{1/4}s)}{Vol B(x, t^{1/4})} (1+s)^\frac{n}{2} e^{- c s^{4/3} } ds \\
    = & -C \int _0^\infty \frac{Vol B(x, t^{1/4} s)}{Vol B(x, t^{1/4})} \frac{d}{ds} \left( (1+s)^\frac{n}{2} e^{- c s^{4/3} } \right)ds. \\
    \end{split}
  \]
  By the volume comparison theorem, 
  \[
    \frac{Vol B(x, t^{1/4} s)}{Vol B(x, t^{1/4})} \leq C(n ) s^n
  \]
  when $s \geq 1$. Hence $\int_M |b(x, y, t)| dt \leq C(n, \mu, v)$, which yields the claim of the theorem. 
\end{proof}

\subsection{$L^\infty$ estimate on closed manifolds}
Although we have focused on complete noncompact manifolds, the results applies also to closed manifolds, i.e. compact manifolds without boundary. In this subsection we discuss the $L^\infty$ estimate for the biharmonic heat equation on closed manifolds since the proof is slightly different from the noncompact case.  
\begin{lem}\label{estimate of biharmonic heat kernel on closed manifolds}
Let $M$ be a closed manifold with dimension $n$, suppose it has diameter $L$, $Ric \geq -K$ and $Vol(B(x,1)) \geq v > 0$ for any $x \in M$. Then 
\begin{equation}
|b(p,q,t)| \leq \frac{C(n, K, v) } {\sqrt{Vol(B(p,t^\frac{1}{4})) Vol(B(q,t^\frac{1}{4}))}} e^{-\frac{c(n, K, v) d(p,q)^{4/3}}{t^{1/3}} },
\end{equation}
for any $p, q \in M$ and $0 < t \leq \bar{T}(n, K, v)$.
\end{lem}
\begin{proof}
Since the manifold is compact, we need to modify the construction of the distance-like function $f_p(x)$ for each $p\in M$. Assume WLOG that $L >> 1$, $f_p(x)$ can be defined on $B_p(\frac{L}{2}-1)$, then $D_{\frac{L}{2}-1} \subset B_p(\frac{L}{2}-1)$, define $f_p$ to be constant on $M \backslash D_{\frac{L}{2}-1}$. Then the same argument as in the noncompact case will establish the decay estimate for the biharmonic heat kernel. Hence by the proof of Theorem \ref{exponential decay estimate with Ricci lower bound}, there is a constant $\bar{T}(n, K, v)$, such that when $ 0 < t   \leq \bar{T}$, we can take $S = C(n, K, v) t^\frac{1}{4}$ for some constant $C(n, K, v)$, then we have
\begin{equation}\label{estimate of the biharmonic heat kernel on closed manifolds}
|b(p,q,t)| \leq \frac{C(n, K, v) } {\sqrt{Vol(B(p,t^\frac{1}{4})) Vol(B(q,t^\frac{1}{4}))}} e^{-\frac{c(n, K, v) \min \{ d(p,q), L/2 -1 \}^{4/3}}{t^{1/3}} },
\end{equation}
for any $p, q \in M$.  Since we have assumed $L >>1$, clearly
$ \min \{ d(p,q), L/2 -1 \} > \frac{1}{3} d(p,q)$ when $d(p, q) > \frac{L}{2} - 1$, hence we have the claimed estimate of the lemma. 
\end{proof}
\begin{cor}\label{cor: l-infinite estimate on closed manifolds}
Let $M$ be a closed manifold with dimension $n$, suppose it has diameter $L$ and Ricci lower bound $Ric \geq -K$, and $Vol(B(x, 1)) \geq v > 0$ for any $x\in M$. Then there is a constnat $C(n, K, v)$, such that for any solution $u(x, t)$ of the biharmonic heat equation, we have
\[
|u(x, t)|_{L^\infty(M)} \leq C(n,K,v) e^{C(n)\sqrt{K}L} |u(x, 0)|_{L^\infty(M)},
\]
for any $t>0$.
\end{cor}
\begin{proof}
WLOG we assume that $L >>1$, otherwise we can scale up the metric and clearly $Ric \geq -K $ still holds. 
By the proof of Lemma \ref{lem: l2 estimates for poly-Laplacian} without using cut-off functions, we can show that $\int_M u^2(x, t)$ is nonincreasing in $t$, and  
\[
\int_M |\Delta^m u (x, t)|^2 \leq \frac{C(n,m)}{t^m} \int_M u(x, 0)^2.
\]
Then by the multiplicative Sobolev inequality Lemma \ref{lem: l2 multiplicative sobolev inequality} (where we can take $r$ to be the diameter $L$) , we have 
\[
|u(x, t)|_{L^\infty(M)} \leq  C(n, \mu) e^{C(n, \mu) \sqrt{K}L } (1+\frac{1}{t})^{C(n, \mu)}\mnorm{ u(x, 0) } _{L^2(M)}, \quad t > 0.
\]
This proves the corollary when $t$ is large.

 For $t$ small we use the biharmonic heat kernel estimate in Lemma \ref{estimate of the biharmonic heat kernel on closed manifolds}. 
 By the volume comparison theorem
\[
Vol(B(p, t^\frac{1}{4}) )  \leq e^{C(n ) \sqrt{K} L} \left(1 + \frac{d(p,q)}{t^\frac{1}{4}} \right)^n Vol (B(q, t^\frac{1}{4})).
\] 
Thus
\[
|b(p,q,t)| \leq \frac{C(n, K, v) e^{C(n ) \sqrt{K} L} } { Vol(B(p,t^\frac{1}{4})) } \left(1 + \frac{d(p,q)}{t^\frac{1}{4}} \right)^\frac{n}{2}  e^{-\frac{c(n, K, v) d(p,q)^{4/3}}{t^{1/3}} }.
\]

Suppose $u(x, t)$ is a solution of the biharmonic heat equation on $M$. By the uniqueness Theorem \ref{uniqueness theorem},
\[
u(x, t) = \int_M b(x, y, t) u(y,0) dy.
\]
By the same argument as in the proof of Corollary \ref{cor: uniform l-infinity estimate}, we can show that
\begin{equation}\label{integrability of biharmonic heat kernel on closed manifolds}
\int_M |b(x, y, t)| dy \leq C(n, K, v) e^{c(n)\sqrt{K}L} , \quad \forall (x, t) \in M \times (0, \bar{T}).
\end{equation}
Hence
\[
|u(x, t) | \leq C(n, K, v) e^{c(n)\sqrt{K}L} |u(x, 0)|_{L^\infty(M)}, \quad \forall (x, t) \in M \times (0, \bar{T}).
\]
\end{proof}

\section*{Appendix}
The following is an example of a complete Riemannian manifold, which admits a solution of the biharmonic heat equation, starting with bounded initial data, but its $L^\infty(M)$ norm goes to $\infty$ as $t \to \infty$. The construction was inspired by examples of similar form in \cite{SNWC1977}.

\textbf{Example.}
Let $M = \mathbb{R} \times N^{n-1}$ with a complete metric   
\[ds^2 = dr^2 + \phi^2(r) ds_{N}^2,\]
where we take 
\[\phi(r) = e^{r^{2+\epsilon}}, \quad \epsilon > 0.\]
The Laplacian of $M$ applied to a function $f(r)$ can be written as 
\[
\Delta f(r) = \phi^{1-n} (r) \left(\phi^{n-1}(r) f'(r)\right)'.  
\]
Let
\[
  F(r) = \int_0^r \phi^{1-n} (s) \int_0^s \phi^{n-1} (\tau) \int_0^\tau \phi^{1-n}(\gamma) \int_0^\gamma \phi^{n-1}(\eta) d\eta d\gamma d\tau ds,
\]
then $F$ is a bounded function,
 such that $\Delta^2 F = 1$. Define $v(r,t) = F(r) - t$, then $v$ is a solution of the biharmonic heat equation, which is bounded at $t=0$, but its $L^\infty$ norm goes to $\infty$ as $t \to \infty$.


\bibliographystyle{plain}

\end{document}